\documentclass[12pt,a4paper]{amsart}

\usepackage[margin=1.1in]{geometry}
\usepackage{MnSymbol,accents,mathrsfs,upref}
\usepackage{hyperref}
\usepackage[english]{babel}

\expandafter\let\expandafter\oldproof\csname\string\proof\endcsname
\let\oldendproof\endproof
\renewenvironment{proof}[1][\proofname]{  \oldproof[\bfseries #1]}{\oldendproof}

\theoremstyle{plain}
\newtheorem{Theorem}{Theorem}
\newtheorem{Proposition}{Proposition}[section]
\newtheorem{Lemma}[Proposition]{Lemma}
\newtheorem{Corollary}[Proposition]{Corollary}
\newtheorem*{Newton-Lagrange}{Newton-Lagrange interpolation formula}

\theoremstyle{definition}
\newtheorem{Definition}{Definition}
\newtheorem{Hypothesis}{Hypothesis}

\theoremstyle{remark}
\newtheorem{Remark}[Proposition]{Remark}
\newtheorem{Example}{Example}
\newtheorem*{Notations}{Notations}

\numberwithin{equation}{section}
\newcommand{\beq}{\begin{equation}}
\newcommand{\eeq}{\end{equation}}
\newcommand{\sm}{\smallskip}
\newcommand{\md}{\medskip}

\newcommand{\ug}{\,=\,}
\newcommand{\eqv}{\,\equiv\,}
\newcommand{\mn}{\,\le\,}
\newcommand{\mg}{\,\ge\,}
\newcommand{\noi} {\noindent}
\newcommand{\de}{\partial}
\newcommand{\+}{\,+\,}

\newcommand{\pv}{\,;\,}
\newcommand{\gb}{\goodbreak}

\newcommand{\R}{\mathbb{R}}

\newcommand{\wh}[1]{\widehat{#1}}
\newcommand{\wt}[1]{\widetilde{#1}}

\let\Im\undefined
\DeclareMathOperator{\Im}{Im}
\newcommand{\transpose}[1]{{#1}^T}

\DeclareMathOperator{\Span}{Span}

\let\Mat\undefined
\DeclareMathOperator{\Mat}{Mat}

\newcommand{\PP}{\mathsf{P}}
\newcommand{\LL}{\mathsf{L}}
\newcommand{\RR}{\mathsf{R}}
\renewcommand{\SS}{\mathsf{S}}
\newcommand{\rr}{\mathsf{r}}

\newcommand{\taud}[2][d]{\tau^{(#1)}_{#2}}
\newcommand{\taumax}{\tau_{\mathrm{max}}}

\newcommand{\uu}[1]{\mathsf{u}_{#1}}
\newcommand{\vv}[1]{\mathsf{v}_{#1}}

\newcommand{\uuu}{\mathsf{u}}
\newcommand{\vvv}{\mathsf{v}}
\newcommand{\fff}{\mathsf{f}}

\newcommand{\wwm}[1]{\mathsf{w}_{#1}}

\newcommand{\sA}{\mathscr{A}}
\newcommand{\sB}{\mathscr{B}}

\newcommand{\sQ}{\mathscr{Q}}

\newcommand{\sE}{\mathscr{E}}

\newcommand{\sN}{\mathcal{N}}

\newcommand{\e}{\varepsilon}

\begin{document}

\title[The Cauchy Problem for properly hyperbolic equations \dots]
   {The Cauchy Problem\\
    for properly hyperbolic equations\\
    in one space variable}
\author{Sergio Spagnolo \textit{\&} Giovanni Taglialatela}
\dedicatory{Dedicated to our friend Enrico Jannelli}

\address{   Dipartimento di Matematica,
   Universit\`{a} di~Pisa,
   Largo Bruno Pontecorvo 5,
   56127 Pisa,
   Italy}
\email{sergio.spagnolo@unipi.it}
\address{
   Dipartimento di Economia e Finanza,
   Universit\`{a} di~Bari ``Aldo Moro'',
   Largo Abbazia S. Scolastica,
   70124 Bari,
   Italy}
\email{giovanni.taglialatela@uniba.it}
\maketitle

\tableofcontents

\section{Introduction}

In this paper we~deal with the~non characteristic Cauchy Problem
\begin{align}  \qquad \mathcal{L}\,u
&     \ug f(t,x)  ,  &
   &  t\in\R^+ , \ x\in\R  ,
      \label{E-scalar}  \\
\qquad \partial_t^j u(0,x)
&  \ug  \varphi_j(x),
\quad \ 0 \le j \le m-1  ,  &
   &  x\in\R ,
         \label{E-data}
\end{align}
where  $\mathcal{L}$\, is~an operator of~order~$m\ge2$ in $\,\R_t\times \R_x$, with smooth coefficients, of the form
\begin{equation} \label{eq:L}
 \mathcal{L}\,u  \eqv \LL(t,x\pv\partial_t,\partial_x)\,u
  \ug  \PP(x\pv\partial_t,\partial_x)\,u
           \,-\,  \sum_{d=0}^{m-1} \RR_d(t,x\pv\partial_t,\partial_x)\,u.
\end{equation}

\begin{Hypothesis} \label{H-H} \
The~principal symbol
\begin{equation} \label{eq:princ-symb-0}
\PP(x\pv\tau,\xi)
  \equiv \tau^m
        \+ a_1(x) \, \tau^{m-1} \, \xi
        \+  \dotsb
        \+ a_m(x) \, \xi^m
\end{equation}
depends only on the~spatial variable $x$,
and is~\emph{weakly hyperbolic},
i.e.,
\begin{equation} \label{E-Pm}
\PP(x\pv\tau,\xi)
  \ug  \prod_{j=1}^m\, \bigl(\tau-\tau_j(x)\,\xi\bigr)
\end{equation}
where the~characteristic roots~$\tau_j(x)$
are real functions, not necessarily distinct.
\end{Hypothesis}

The lower order terms
\begin{equation}\label{eq:lot}
\RR_d(t,x\pv\tau,\xi)
  \eqv \sum_{k=0}^d\, r_{d,k}(t,x)\,\tau^{d-k}\,\xi^k , \qquad d=0,1,\dotsc, m-1,
\end{equation}
are homogeneous polynomials in~$(\tau,\xi)$ of~degree~$\,d$,
depending on both variables. \sm

\begin{Notations}
Throughout this paper we shall write shortly:
\begin{equation} \label{eq:P}
\PP(x\pv\tau) \eqv \PP(x\pv\tau,1), \quad \RR_d(t,x\pv\tau) \eqv \RR_d(t, x\pv\tau,1)
\end{equation}
\end{Notations}
\gb

Let $\,\mathcal{S}_0$ be the~initial line $\{t=0\}$.
We say that
the Cauchy Problem~\eqref{E-scalar}-\eqref{E-data}
is~\emph{locally well-posed} (in~$\mathscr{C}^\infty$)
at a point $y_0\in \mathcal{S}_0$
if, for every neighborhood~$\mathcal{I}$
of~$y_0$ in~$\R^+\times\R$,
there exists a neighborhood~$\mathcal{I}'\subseteq \mathcal{I}$,
with $\mathcal{I}'\cap \mathcal{S}_0 = \mathcal{I}\cap \mathcal{S}_0$,
such that
for all $\varphi_j \in \mathscr{C}^\infty(\mathcal{I}\cap \mathcal{S}_0)$
and~$f \in \mathscr{C}^\infty(\mathcal{I})$
there is~a unique solution
$u\in \mathscr{C}^\infty(\mathcal{I}')$.
When, for each strip $\mathcal{I}=[0,T[\times \R$,
we can take $\mathcal{I}'=\mathcal{I}$, we say that the Problem is
\emph{globally well-posed}.
\sm

In the \emph{strictly hyperbolic} case, i.e.,
when~the~characteristic roots are distinct,
the Cauchy Problem~\eqref{E-scalar}-\eqref{E-data}
is well-posed with no further assumption
on the~lower order terms.
On the other side, in case of multiple characteristic roots
some additional conditions are needed,
both on the~principal symbol and on the~lower order terms.
\sm

In~\cite{ST-JHDE} we considered the case of an \emph{homogeneous operator},
that is when $\,\RR_d\equiv0\,$ for all~$d$, and we~proved the
well-posedness for~\eqref{E-scalar}-\eqref{E-data}
under the~following condition:

\begin{Hypothesis} \label{H-CO} \
There exists a constant  $M>0$
such that
\begin{equation} \label{E-CO}
\tau_j^2(x) \+ \tau_k^2(x)
  \mn M\,\bigl(\tau_j(x) \,-\,\tau_k(x)\bigr)^2 \, ,
\qquad
1\le j<k\le m .
\end{equation}
\end{Hypothesis}

Note that, by~Newton's Theorem on symmetric polynomials,
\eqref{E-CO} can be explicitly expressed
in terms of~the coefficients~$a_j(x)\,$
of the principal symbol
(see~\cite[Remark~1.2]{ST-JHDE} for some examples).
\sm

In order to~treat the case of non homogeneous operators,
some  conditions on~the lower order terms
are needed,
even if~the~coefficients are constant.
Indeed,
in the~constant coefficients case
the following condition on the~full symbols~$L(\tau,\xi)\,$ is~\emph{necessary and sufficient}
for the Problem~\eqref{E-scalar}-\eqref{E-data}
to be well-posed in~$\mathscr{C}^\infty$~(see \cite{Garding} \cite{Hormander} \cite{Svensson}): 
\begin{equation} \label{E-Garding}
\text{there exists~$\,C>0\,$ such that
   $\,L(\tau,\xi)\ne0\,$ if~$\,\xi\in\R\,$ and $\,|\Im\tau|>C$} \, .
\end{equation}

This  is the classical G\aa rding condition. 
Various conditions equivalent  to~\eqref{E-Garding} have been found.
We consider here the~formulation given by~Peyser~\cite{Peyser-1963}.
In order to~explain this condition, we give
the following definition:

\begin{Definition} \label{D-pd}
If $\,\PP(x\pv\tau) \equiv  \bigl(\tau-\tau_1(x)\bigr) \cdots
\bigl(\tau-\tau_m(x)\bigr)$
is a hyperbolic polynomial,
we~say that a polynomial $\,\RR(x\pv\tau)$ of~degree~$\,\le m-1\,$
has a~\emph{proper decomposition w.r.t.~$\,\PP(x \pv \tau)$}
if~there exist some functions $\,\ell_k\in L^\infty(\R)$ such that
\begin{equation} \label{E-dec ell 0}
\RR(x\pv\tau)
  \ug  \sum_{k=1}^m \ell_k(x) \, \PP_{\wh{k}}(x\pv\tau) ,
  \qquad \text{for all} \ x\in\R,
\end{equation}
where
 \begin{equation} \label{eq:reduced}
\PP_{\wh{k}}(x\pv\tau)
  \eqv  \prod_{\substack{j=1,\dotsc,m \\ j\ne k}}
      \bigl(\tau-\tau_j(x)\bigr) , \quad \ k=1,\ldots,m.
\end{equation}
The polynomials $\,\PP_{\wh{k}}\,$ are  called the~\emph{reduced}
(or~\emph{incomplete}) polynomials of~$\,\PP$.
\end{Definition}
\sm

\begin{Example} \label{Ex-dtPP}
The polynomials
$\,\partial_\tau \PP(x\pv\tau)$ and $\,\partial_x \PP(x\pv\tau)$
have a proper decomposition w.r.t.~$\PP(x\pv\tau)$.
The first assertion is obvious, indeed
\begin{equation} \label{E-Pp Ph}
\partial_\tau \PP(x\pv\tau)
  \ug  \sum_{h=1}^m \, \PP_{\wh{h}}(x\pv\tau) \, .
\end{equation}
On the other side, recalling the~Lemma of~Bron\v{s}te\u{\i}n
on  Lipschitz continuity of~the characteristic roots
of an~hyperbolic polynomial
(\cite{Bronstein-1979};
cf.~\cite{Wakabayashi-1986} \cite{Mandai-1985} \cite{Tarama-2006}),
we find
\begin{equation} \label{E-Px Ph}
\partial_x \PP(x\pv\tau)
  \ug -\sum_{h=1}^m \, \tau_h'(x) \, \PP_{\wh{h}}(x\pv\tau) \, ,
\end{equation}
so that $\,\partial_x \PP$ has a proper decomposition w.r.t.~$\PP$.
\end{Example}
 \gb

\begin{Remark} \label{R-1}
If, for some fixed $\,x\in\R$, the polynomial $\,\PP(x\pv\tau)$
has simple (real) roots, the $\ell_k$s in~\eqref{E-dec ell 0}
are uniquely determined by Newton-Lagrange formula:
\[
\ell_j(x)
  =  \frac{\RR\bigl(x\pv\tau_j(x)\bigr)}{\PP_{\wh{k}}\bigl(x\pv\tau_j(x)\bigr)} \, .
\]  
\end{Remark}

In this paper, recalling \eqref{eq:P}, we assume:

\begin{Hypothesis} \label{H-ph} \
For each $\,d=1,\dotsc,m-1$, the polynomials $\,\RR_d(t,x\pv\tau)$
have a proper decomposition w.r.t.~$\,\partial_\tau^{m-1-d}\, \PP(x\pv \tau)$.
\end{Hypothesis}

This means that
$\,\RR_d(t,x\pv\tau)$
is a linear combination of~the reduced polynomials of the polynomial
$\,\partial_\tau^{m-1-d} \,\PP(x\pv\tau)$,
with coefficients  uniformly bounded  with respect to $\,t, x$.
\sm

Following Peyser~\cite{Peyser-1963}
we call~\emph{properly hyperbolic} the operators  verifying Hypothesis~\ref{H-ph}.
\sm

In the constant coefficients case, Peyser~\cite{Peyser-1963} derived suitable~energy estimates
ensuring the $\mathcal{C}^\infty$ well-posedness  for the operators that verify Hypothesis~\ref{H-ph}.
Dunn~\cite{Dunn} extended the~result of~Peyser
proving that Hypothesis~\ref{H-ph}
is~sufficient to ensure the well-posedness also if~the~lower order terms
have variable coefficients.
On the other side Wakabayashi~\cite{Wakabayashi-1980}
proved that the~proper hyperbolicity is~also necessary.
\sm \gb

All of~the previously mentioned results concern the case of
operators with constant coefficient principal part. 
In this paper we consider operators
with variable coefficients principal part
satisfying Hypothesis~\ref{H-CO}, and we prove:

\begin{Theorem} \label{T-1}
Let
$\,a_k \in \mathscr{C}^\infty(\R)$
for $\,1\le k\le m$,   $\,r_{d,k} \in L^\infty\bigl([0,T] \times \R\bigr)$
for $\,0\le k\le d\le m-1$.
Then, if~Hypothesis~\ref{H-H}, \ref{H-CO} and~\ref{H-ph}
are satisfied,
the~Cauchy Problem~\eqref{E-scalar}-\eqref{E-data}
is locally well-posed.
If in addition  the $\,a_k(x)$'s are  bounded on $\R$,
we have  the global well-posedness.
\end{Theorem}

\begin{Remark}
An analogous result holds true for operators
with principal symbol depending only on the~time variable.
Condition~\eqref{E-CO} was considered
for the~first time in~\cite{ColombiniOrru1999},
where the~$\,\mathcal{C}^\infty$ well-posedness was proved
for homogeneous equations with analytic-type
coefficients depending only on time.
The result was then extended in~\cite{KinoshitaSpagnolo}
to operators with non-analytic coefficients
and in~\cite{T-2013}
and~\cite{GarettoRuzhansky-2013} to non-homogeneous operators.
\end{Remark}

\begin{Remark}
Our result is restricted to operators in one space dimension,
since Hypothesis~\ref{H-CO} is no longer sufficient
for $\mathcal{C}^\infty$ well-posedness in the higher-dimensional case.
Indeed, in~\cite{BernardiBove-2006}
they prove that the Cauchy problem for the equation
\[
\partial_{tt}^2 u
   \+  a_1 (x,y,\partial_x,\partial_y) \partial_t u
   \+  a_2 (x,y,\partial_x,\partial_y) u \ug 0 \, ,
\]
where
\[
a_1(x,y\pv\xi,\eta)
  \ug  -2x\eta \, ,
\qquad
a_2(x,y\pv\xi,\eta)
  \ug  -\xi^2 + x^3 \eta^2 \, ,
\]
is not locally well-posed at the origin,
although this equation fulfills the analogue
of Hypothesis~\ref{H-CO} near $(x,y) = (0,0)$,
viz.
\[
a_1^2(x,y\pv\xi,\eta)
  \mn  C\,\Big\{(a_1^2(x,y\pv\xi,\eta)
                 - 4\,a_2(x,y\pv\xi,\eta)\Bigr\} ,
\quad
\text{for all $(x,y\pv\xi,\eta) \in \R^2 \times \R^2$} \, .
\]
\end{Remark}
\gb

\textbf{Sketch of the proof of Theorem \ref{T-1}.} \

Firstly, we transform equation \eqref{E-scalar} into a first order system.
Then, resorting to Jannelli's symmetrizer~\cite{Jannelli1989} \cite{Jannelli2009}
and to the Hypotheses~B and~C, we get an apriori energy estimate.
Finally, using  a form  of  the Lemma of Nuij~\cite{Nuij},
we approximate \eqref{E-scalar} by a sequence of strictly hyperbolic equations,
thus obtaining a sequence of approximating solutions which  converge to a solution of  \eqref{E-scalar}.
\sm

To make the proof  more understandable,
in the next section we shall briefly recall the \emph{homogeneous case},
i.e., that of equation~\eqref{eq:L} with $\,\RR_d \eqv 0$ for all $\,d$.
This case was already treated in~\cite{ST-JHDE}.
Then, in \S 3 and \S 4, we shall treat the general case.
\gb

\begin{Remark}
The (standard) symmetrizer is a useful tool to study homogeneous operators
\cite{ST-JHDE} \cite{JannelliTaglialatela-2011}
or~operators in which the behaviour is essentially
determined by the principal symbol
\cite{Nishitani2021}, \cite{NishitaniPetkov1}, \cite{NishitaniPetkov2}.
 A~more powerful tool to treat operators
with non smooth coefficients~\cite{JannelliTaglialatela-2014},
Levi conditions~\cite{T-2013}, \cite{GarettoRuzhansky-2013},
or non-linear problems~\cite{Spagnolo-2005}, \cite{ST-CPDE},
is given by the so-called \emph{quasi-symmetrizer},
introduced in~\cite{DAnconaSpagnolo1998} and extensively studied in~\cite{Jannelli2009}
(see~also~\cite{T-2013} and~\cite{Nishitani2020}).
\sm

The quasi-symmetrizer is a perturbation of the standard symmetrizer:
\[
Q_\e \ug Q_0 \+ \e^2 Q_1 \+ \dotsb\ + \e^{2(m-1)} Q_{m-1}
\]
where $\,Q_0$ is the standard symmetrizer of $\,\PP$,
$\,Q_1$ is the sum of the symmetrizers of the reduced polynomials~$\,\PP$
(cf.~\eqref{eq:reduced}),
$\,Q_2$ is the sum of the symmetrizers of the bi-reduced polynomials~$\,\PP$
(see~\eqref{E-bired} below)
and so on.
\sm

In this paper we use a different approach to modify the symmetrizer:
we~construct a symmetrizer
\[
\sQ \ug Q_0 \oplus Q^{(1)} \oplus \dotsb \oplus Q^{(m-1)} \, .
\]
where $\,Q_0$ is still the standard symmetrizer of $\,\PP$,
whereas $\,Q^{(1)}$ is the standard symmetrizer of $\,\partial_\tau \PP$,
$\,Q^{(2)}$ is the standard symmetrizer of $\,\partial_\tau^2 \PP$,
and so on.
\end{Remark}
\gb

\section{The homogeneous case  and the Jannelli's symmetrizer}

In this section we  assume $\,\RR_d \eqv 0\,$ for all $\,d=0, \cdots, m-1$.
Thus, equation \eqref{E-scalar} becomes
\sm
\begin{equation} \label{E-hom.eq.}
 \PP(x\pv\partial_t,\partial_x)\,u \eqv \partial_t^mu
        \,+\,  a_1(x)\,\partial_t^{m-1}\partial_x u\,
        +\,  \dotsb
        \,+\, a_m(x) \,\partial_x^{m}u
 \ug f(t,x).
\end{equation}
\sm
Now we introduce the~$m$-vectors:
\begin{equation}\label{eq:u-f}
\uuu(t,x)  :\ug
\begin{pmatrix}
      \partial_x^{m-1} u \\
      \partial_t \partial_x^{m-2} u \\
      \vdots \\
      \partial_t^{m-1} u,
      \end{pmatrix} \,,
       \qquad
  \fff(t,x):\ug  \begin{pmatrix}
      \,0\, \\
   \, 0\,  \\
      \vdots \\
    \,f\,
      \end{pmatrix}\,,
\end{equation}
and the $\,m\times m$ \emph{Sylvester matrix}
associated to~the~polynomial $\,\PP(x,\tau)$:
\begin{equation} \label{E-Sylvester}
A(x)
  \eqv  \begin{pmatrix}
       0 & 1 & \cdots & 0  \\
       \vdots &  & \ddots & \vdots  \\
         & &  & 0  \\
       0 & \cdots & 0 & 1  \\
       -a_m(x) &  -a_{m-1}(x) & \cdots & -a_1(x)
      \end{pmatrix} \,
              \end{equation}
\sm
Clearly, the equation \eqref{E-hom.eq.} is equivalent to the system
\begin{equation} \label{E-system m}
\uuu_t
  \ug  A(x) \, \uuu_x
      \,+\, \fff(t,x) .
\end{equation}
To get an~apriori estimate for the solutions to this system,
we resort to the  energy function based on~the
Jannelli's symmetrizer~\cite{Jannelli1989} \cite{Jannelli2009}.
We recall that a \emph{symmetrizer} of a $\,m\times m\,$  matrix $\,A$ is a  $m\times m\,$ matrix~$\, Q$
such that
\[
 Q^T \ug  Q \mg\, 0, \qquad ( Q A)^T \ug  Q A.
\]
Correspondingly to a given symmetrizer,    we define  the energy of a  solution  of \eqref{E-system m} as:
\begin{equation}  \label{E:energy}
E(t,\uuu) : \ug  \frac{1}{2}
      \int_{I_t} \bigl( Q(x)\, \uuu(t,x) \,, \,\uuu(t,x)\bigr)\,dx,
\end{equation}
where $\,I_t$ is a suitable real interval, and we look for an apriori estimate.
\sm

In the  case that the matrix $ A$ (hence also $\, Q\,$) is constant, and $\,I_t =\R$, by differentiating~\eqref{E-system m} in $t\,$ we find
\[
E'(t,\uuu)   \ug  \frac{1}{2}
      \int_{\R}\big\{ \bigl( Q  A\uuu_x, \uuu\,\bigr) + \bigl( Q \uuu,  A\,\uuu_x\bigr)\big\}\,dx
      \ug  \int_{\R}\ \de_x\bigl( Q  A\uuu, \uuu\,\bigr) \,dx \ug 0
\]
for any  solution $\,\uuu(t,\cdot)$ with compact support in $\,\R$.
Thus, $\,E(t,\uuu)$ keeps constant.
\sm

In the general case, in order to get an estimate of $\,E'(t,\uuu) $ in term of $E(t,\uuu) $, we have to~estimate   $\,Q A$ and $\,\de_x(\,Q A)$ in terms of $\,Q$. Moreover, in order to derive an estimate of~$\uuu$, we need some {\it coercivity\,} of the symmetrizer $\,Q(x)$.
The Jannelli's symmetrizer, under Hypothesis~\ref{H-CO},
enjoys of all these properties (Proposition~\ref{P-Q} below).
\sm

We now recall the construction of such a symmetrizer
(cf.~\cite{Jannelli1989}, \cite{Jannelli2009}):
First of all we introduce a notation.
\sm

\begin{Notations}
Given a polynomial $\,p(\tau) \eqv c_0+c_1\tau+\cdots+c_d\,\tau^d\,$
of~degree~$d$, we define the row vectors:
\begin{align} \label{eq:vec}
\mathsf{vec} (p)
       & :\ug (c_d,\dotsc,c_1,c_0) \in\, \R^{d+1}
        \\
  \label{E-wk}
  \wwm{k}(x)
 & :\ug  \mathsf{vec}\,\bigl(\PP_{\wh{k}}(x\pv\cdot)\bigr) \in \,\R^m, \qquad k=1,\dotsc,m,
\end{align}
where the $\,\PP_{\wh{k}}\,$ are the reduced polynomials of~$\,\PP$,
and we introduce the $m\times m$ matrix
\[
 W(x)
  :\ug   \begin{pmatrix}
     \ \wwm{1}(x) \ \\ \vdots \\ \ \wwm{m}(x) \    \end{pmatrix}.
     \]
\end{Notations}

If $\,\{\tau_1(x) \le \cdots \le \tau_m(x)\}$ are the eigenvalues of the matrix $\,A(x)$,
that is the roots of~$\,\PP(x,\tau)=0$,
from the above definition we easily derive (see \cite{ST-JHDE}) that,
up to~dilatations, the row vector $\,\wwm{k}(x)$ is the \emph{unique left eigenvector} of~$\, A(x)$
corresponding to~the eigenvalue $\tau_k(x)$, i.e.,
\begin{equation} \label{eq:eigenvectors}
\wwm{k}(x)  A(x) \ug \tau_k(x)\, \wwm{k}(x) .
\end{equation}

\begin{Definition} \label{D-J}
We call \emph{Jannelli's} (or {\it standard\,}){\,\it  symmetrizer}  of $\,A(x)$ the $m\times m$ matrix
\begin{equation} \label{eq:Jsym}
 Q(x)
  :\ug  W^T(x) \,  W(x).
\end{equation}
\end{Definition}

Note the identity:
\begin{equation} \label{E-wcdot}
\bigl( Q(x)\, \vvv,\vvv\bigr)
  \ug  \sum_{k=1}^m \,\bigl(\,\wwm{k}(x) , \vvv\,\bigr)^2 ,
\qquad
\text{for all} \  \vvv\in\R^m, \ x\in \R.
\end{equation}
\gb

Later on, we shall use the following  characterisation  of the
proper  decomposition:

\begin{Lemma} \label{E-equi}
A~polynomial $\,\ \RR(x\pv\tau)$ of~degree~$\,\le m-1\,$
has a~proper decomposition w.r.t. $\,\PP(x\pv \tau)$
\textup(see \eqref{eq:princ-symb-0}-\eqref{eq:P}\textup)
if and only if there exists some constant $\,C_0$ for which
\begin{equation} \label{E-Ru}
\bigl(\mathsf{vec}\bigl(\RR(x \pv \cdot)\bigr), \vvv\,\bigr)^2
  \mn  C_0 \, \bigl( Q(x)\, \vvv , \vvv\bigr) \, ,
\qquad
\text{\rm for all} \, \ \vvv\in\R^m, \ x\in \R.
\end{equation}
\end{Lemma}

\begin{proof}
If we define the vector
\[
\rr(x): \ug \mathsf{vec}\bigl(\RR(x \pv \cdot\bigr),
\]
and  recall~\eqref{E-dec ell 0}-\eqref{E-wk}, we see that $\,\RR$ has a~proper decomposition w.r.t.~$\,\PP$ if and only if there are 
some $\,\ell_k\in L^\infty(\R)$ for which
\begin{equation} \label{eq:r-wk}
\rr(x) \ug 
    \sum_{k=1}^m \ell_k(x) \, \wwm{k}(x), \qquad \text {for all}\ x\in \R.
\end{equation}
\gb 
Hence, if  $\,\RR$ has a~proper decomposition w.r.t.~$\,\PP$, by \eqref{eq:r-wk}   and~\eqref{E-wcdot} we derive:
\[
\bigl(\rr(x),\, \vvv\,\bigr)^2
  \mn   \sum_{k=1}^m \,\ell_{k}^{\,2}(x) \cdot \sum_{k=1}^m \bigl(\wwm{k}(x) \,, v\bigr)^2
  \mn  C \, \bigl( Q(x) \vvv , \vvv\bigr),
\]
with C $\,\ug \sum \|\ell_k\|_{L^\infty}^2$.
\sm

Conversely, let us assume~\eqref{E-Ru}. If we define, for each fixed $\,x\in \R$, the linear space
\[
Z(x):\ug \Span\bigl\{\wwm{1}(x),\dotsc,\wwm{m}(x)\bigr\} \subset \R^m
\]
and note that  $\,\bigl( Q(x)\vvv, \vvv\bigr)=0\,$
if  $\,\vvv\in Z(x)^\bot$ (see~\eqref{E-wcdot}), we get
\[
\bigl(\rr(x) , \vvv \bigr) \ug 0
\quad \text{for all }\, \vvv \in Z(x)^\bot.
\]
\sm
Consequently we see that $\,\rr(x) \in Z(x)$. Thus, there are~$\,\ell_1(x), \ldots, \ell_m(x)$ such that
\begin{equation} \label{E-23}
\rr(x) \ug \sum_{k=1}^m \ell_k(x) \wwm{k}(x),
\end{equation}
where, after a slight mdification, we can assume that  $\,\ell_h=\ell_k$ whenever  $\,\wwm{h}=\wwm{k}$.
\sm

We prove that
\begin{equation} \label{E-27}
\bigl|\ell_{k}(x)\bigr| \mn \sqrt{C_0\,} \, , \qquad
\text{for each fixed} \  k  \ \text{and} \ x.
\end{equation}
Indeed, if the eigenvalue $\,\tau_k(x)$ has multiplicity $\,\nu$ $ (1\le \nu \le m)$,
let  $\,h_1,\dotsc,h_\nu\,$ be the indices $\,h$ for which
$\,\tau_h({x}) = \tau_{k}({x})$,
and consequently
$\wwm{h}({x}) = \wwm{k}({x})$ and $\,\ell_h(x)=\ell_k(x)$.
Therefore  taking a vector~$\,\overline{\vvv}\in \R^m$
which satisfies $\,\bigl(\wwm{k}({x}) ,  \overline{\vvv}\bigr) =1$,
and
$\,\bigl(\wwm{h}({x}) ,  \overline{\vvv}\bigr) =0$
if~$\,h \notin \{ h_1,\dotsc,h_\nu \}$,
and recalling~\eqref{E-23}, we~find:
\begin{equation} \label{E-24}
\bigl(\rr(x)\,,\, \overline{\vvv}\bigr)
  \ug  \sum_{h=1}^m \ell_h({x}) \bigl(\wwm{h}({x}) \,,\, \overline{\vvv}\bigr)
  \ug  \nu \, \ell_{k}({x}) \, .
\end{equation}

On the other side, by~\eqref{E-wcdot} we~know that
\begin{equation} \label{E-25}
\bigl( Q({x}) \overline{\vvv}, \overline{\vvv}\bigr)
   \ug  \sum_{h=1}^m \,\bigl(\wwm{h}(\overline{x}) ,  \overline{\vvv}\bigr)^2
   \ug  \nu,
\end{equation}
thus, inserting~\eqref{E-24} and~\eqref{E-25} in~\eqref{E-Ru},
we get
\[
\bigl|\ell_{k}(x)\bigr| \mn \sqrt{\frac{C_0}{\nu}\,} \ .
\]
Hence~\eqref{E-27} holds true.
\end{proof}
\gb

Next, putting
\[
\taumax :\ug \max_{x\in \R} \max_{1\le j\le m} |\tau_j(x)|\,, \qquad \psi_k
   :\ug \sum_{1\le j_1<\dotsb<j_k\le m}
       \tau_{j_1}^2\, \dotsb\ \tau_{j_k}^2\,,
       \]
 we define the  diagonal matrices:
  \begin{equation} \label{eq:Lambda-Psi}
  \Lambda(x)
  :\ug  \begin{pmatrix}
      \tau_1& 0 & \cdots & & 0 \\
      0 & \tau_2 & 0 & & \vdots \\
      \vdots & & \ddots & \ddots& \\
      0 &\cdots & &\tau_{m-1}& 0 \\
      0 &  \cdots &&  & \tau_m
     \end{pmatrix} \, , \qquad
     \Psi(x)
  :\ug  \begin{pmatrix}
    \psi_{m-1}& 0 & \cdots & & 0  \\
    0 &  \psi_{m-2}& 0 &  &  \vdots  \\
          \vdots & &\ddots & \ddots & & \\
      0 & \cdots& & \psi_1 & 0  \\
      0 & \cdots & &  & 1
      \end{pmatrix} \, .
  \end{equation}
\gb \sm
By \eqref{eq:eigenvectors} we have
\[
 W(x)  A(x) \ug \Lambda(x)  W(x),
\]
whence
\[
 Q  A \ug  W^T \, \Lambda\, W.
\]
Thus, the matrices   $\, Q$ and $\, Q A$ are both symmetric, so  $\, Q$ is a symmetrizer of $\, A$.
\sm

In the following Proposition we resume the  properties of~the Jannelli's symmetrizer \eqref{eq:Jsym}. We  write for brevity $ \, Q \equiv  Q(x), \,  A \equiv A(x),\, \Psi\equiv \Psi(x), \,\tau_j \eqv \tau_j(x)$.

\begin{Proposition} \label{P-Q}
The entries of the symmetrizer $\, Q(x)$ in \eqref{eq:Jsym} are polynomials in the coefficients of~$\,\PP(x \pv\tau)$, in particular  are smooth functions of $x$.
\sm

It holds:
\begin{equation} \label{E-detQ}
\det \, Q
  \ug  \prod_{1\le j<k\le m} \bigl(\tau_j-\tau_k\bigr)^2  .
\end{equation}
Moreover, if the~entries of $\,A(x)$  verify
\begin{equation} \label{E-adk}
\bigl\|a_j^{(k)}\bigr\|_{L^\infty(\R)}
  \le  L
   <   +\infty \, ,
\qquad
\text{for  \ $j,k=0,\dotsc, m$},
\end{equation}
then there exists a positive constant $\,\Gamma_0 \eqv \Gamma_0(m,L)$
such that, for any $\,\vvv\in \R^m$, it holds:
\begin{align}
\label{E-QA}
\bigl|\, Q A\,\vvv,\vvv\,\bigr)\,\bigr|
&  \mn  \taumax\, |\,\vvv\,|^2 ,  \\
 \label{E-QAp}
\bigl|\,\bigl( Q A)'\,\vvv,\vvv\,\bigr)\,\bigr|
&  \mn  \Gamma_0 \, \bigl(\Psi\vvv,\vvv\bigr) ,  \\
 \label{E-Ap}
\bigl|\,\bigl( Q A'\,\vvv,  A'\vvv\bigr)\,\bigr|
&   \eqv   \big| A' \vvv\,|^2
   \le  \Gamma_0 \, \bigl(\Psi\vvv,\vvv\,\bigr) .
\end{align}
Finally, if  Hypothesis~\ref{H-CO} is~satisfied,
the symmetrizer $\, Q(x)$
is a \emph{nearly diagonal matrix},
in the sense that
there exist two positive constants
$\,\Gamma_j \equiv  \Gamma_j(m,M)$, $j=1,2$,
for which
\begin{equation} \label{E-Q sim Psi-x}
\Gamma_1 \, \bigl(\Psi\vvv,\vvv\bigr)
  \mn  \bigl( Q\vvv,\vvv\bigr)
  \mn  \Gamma_2\, \bigl(\Psi\vvv,\vvv\bigr) .
\end{equation}
\end{Proposition}
\sm

\begin{proof}
See \cite[Propositions~1.10 and~1.11]{ST-JHDE}.
\end{proof}

Always in  \cite{ST-JHDE} we define the energy
\begin{equation} \label{eq:energy}
E(t,\uuu)
  :\ug  \frac{1}{2}
      \int_{I_t} \bigl( Q(x) \, \uuu(t,x) \,, \,\uuu(t,x)\bigr)\,dx ,
\end{equation}
where
\begin{equation} \label{E-def It}
I_t
  :\ug  [x_0-\rho(t), \, x_0+\rho(t)],
\qquad
\rho(t)
  : \ug \rho_0-\taumax \, t ,
\end{equation}
and, using the properties \eqref{E-QA}--\eqref{E-Ap} of $\, Q(x)$,
we prove the apriori estimate
\begin{equation} \label{E-E}
E(t,\uuu)
  \mn  C(t)\, \Bigl\{
  E (0, \uuu)
       \+  \int_0^t \!\!\!\int_{I_s} \big|f(s,x)\big| \, dx\, ds
     \Bigr\} \, ,
\quad \ 0\mn t\mn  \rho_0/\taumax \,.
\end{equation}
\gb
\noi {}From \eqref{E-E}  we cannot derive an estimate on $\uuu(t)$:
indeed to his end we need an estimate like $\,\|\uuu(t, \cdot)\| \le C\,E(t,\uuu)$,
which holds only when $\, Q(x)$ is coercive (see \eqref{eq:energy}).
Now, by \eqref{E-detQ} it~follows that~$\, Q(x)$ is~coercive
if and only if  $\,\PP(x \pv\tau)$ is~strictly hyperbolic,
whereas it degenerates at~those points~$x$
where the~characteristic roots are~multiple.
\sm
\gb
However,
thanks to~Hypothesis~\ref{H-CO},
a~weak coercivity estimate holds.
Indeed, taking into account that the last element of the matrix~$\Psi(x)$ is $1$, by \eqref{E-Q sim Psi-x} we derive:
\begin{equation} \label{E-weakcoercivity}
\bigl( Q(x) \, \vvv , \vvv\bigr) \mg \Gamma_1\, |v_m|^2  ,
\qquad
\text{for all} \ \vvv=(v_1,\dotsc,v_m)^T\in\R^m. 
\end{equation}
Hence we get:
\[
\bigl\|\partial_t^{m-1} \uuu(t,\cdot)\bigr\|_{L^2(I_t)} \mn \, \frac{2}{\Gamma_1}\ E(t,\uuu)  ,
\]
which yields, by~integration in $\,t$,
an apriori estimate of~$\,\uuu(t,\cdot)$ in $L^2(I_t)$. \sm

Differentiating repeatidely~\eqref{E-system m} in~$x$,
and using a suitable estimate of~the derivatives
$ A^{(h)}(x)$  in terms of the $\,\psi_j(x)$'s,
we get an estimate of~$\,\uuu(t,\cdot)$ in  $\,{H^k(I_t)}$.

\section{Reduction to~a first order system} \label{S:princ-symb}

Let us go back to an equation of the general type:
\begin{equation} \label{eq:scalar-eq}
\LL(x\pv\partial_t,\partial_x)u  \equiv\, \PP(x\pv\partial_t,\partial_x)u
           \,-\,  \sum_{d=0}^{m-1} \RR_d(t,x\pv\partial_t,\partial_x)\,u
           \ug 0,
\end{equation}
with principal symbol
\begin{equation} \label{eq:princ-symb}
\PP(x\pv\tau,\xi)
  \ug \tau^m
        \+ a_1(x) \, \tau^{m-1} \, \xi
        \+  \dotsb
        \+ a_m(x) \, \xi^m .
\end{equation}

Proceeding as in the homogeneous case,  we~transform such an equation
into an equivalent first order system.
Before to~proceed,
we~fix some notation.

\begin{Notations} \hfill
\begin{itemize}
\item  $\PP(x \pv \tau) : \ug \PP(x \pv \tau,1), \qquad \RR_d(x \pv \tau) : \ug \RR_d(x \pv \tau,1)$.\sm

\item  For each $\,d=1,\dotsc,m$,
  the \emph{monic $\tau$-derivative} of~order
       $\,m-d$ of~$\,\PP(x\pv\tau)$ is:
              \begin{equation}  \label{eq:monic}
       \PP^{(d)}(x \pv \tau)
         :\ug  \frac{d!}{m!} \ \de_{\tau}^{m-d}
              \, \PP(x \pv \tau)
         \ug  \sum_{k=0}^d \,a_{d,k}(x) \, \tau^{d-k}
         \ug  \prod_{\ell=1}^{d}\, \bigl(\tau-\taud{\ell}(x)\bigr).
         \end{equation}

Note that the functions  $\,a_{d,k}(x)$ are, up to constants,
the coefficients $\,a_j(x)$ of the principal symbol \eqref{eq:princ-symb}.
       \sm

\item  For each $\,d=1,\dotsc,m$,
       $\,A_d(x)$ denotes the~Sylvester matrix associated to~$\,\PP^{(d)}(x\pv\tau)$\,:
              \begin{equation}   \label{eq:Ad}
       A_d(x)
         : \ug  \begin{pmatrix}
              0 & 1 & \cdots & 0  \\
              \vdots &  & \ddots & \vdots  \\
                & &  & 0  \\
              0 & \cdots & 0 & 1  \\
              -a_{d,d}(x) &  -a_{d,d-1}(x) & \cdots & -a_{d,1}(x)
             \end{pmatrix} \, .
       \end{equation}
              If $\,d=m$,  $\,A_{m}(x)$ coincides with the matrix $\, A(x)$ considered in \S 2 (see \eqref{E-Sylvester}).\end{itemize}

\noi Let
\[
p(\tau) \eqv c_{0}\tau^{d}+ \dotsc +c_j \, \tau^{d-j}+ \dotsc + c_{d}
\]
be a generic polynomial of~degree~$\,d$.
Therefore:\sm
\begin{itemize}
\item  $\,\mathsf{vec}(p)\,$ is the~row vector in~$\,\R^{d+1}$
       formed by the~coefficients of~$\,p$\,:
              \[
       \mathsf{vec}(p)
         : \ug  (c_d,c_{d-1},\dotsc,c_1,c_0) \in \R^{d+1}.
       \]

\item
     $\,\Mat_j(p)\,$ is the~$j\times(d+1)$ matrix
       obtained from the~null matrix by
       replacing the~entries of~the last line
       with the~coefficients of~$\,p(\tau)$:
              \[
       \Mat_j(p)
         : \ug   \begin{pmatrix}
             0 & \cdots & & & 0  \\
             \vdots & & & & \vdots  \\
               \\
             0 & \cdots & & & 0  \\
             c_d & c_{d-1} & \cdots & c_1 & c_0
             \end{pmatrix} \, .
       \]

\item  Finally, we define the $\,(d+1)\times (d+1)$-matrices
              \begin{gather*}
       \alpha_d(x)
         \eqv\, \Mat_d\bigl(\PP^{(d)}(x \pv\cdot)\bigr)
          :\ug  \begin{pmatrix}
             0 & \cdots & & & 0  \\
             \vdots & & & & \vdots  \\
               \\
             0 & \cdots & & & 0  \\
             a_{d,d}(x) & a_{d,d-1}(x) & \cdots & a_{d,1}(x) & 1
             \end{pmatrix} \,, \\
              \rho_d(t,x)
         \eqv\,  \Mat_m\bigl(\RR_d(t,x \pv\cdot)\bigr)
          : \ug  \begin{pmatrix}
              0 & & \cdots & 0  \\
             \vdots &  &  & \vdots  \\
                 \\
              0 & & \cdots & 0  \\
             r_{d,d}(t,x) & r_{d,d-1}(t.x) & \cdots & r_{d,0}(t,x)
             \end{pmatrix} \, .
       \end{gather*}
\end{itemize}
\end{Notations}
\md

Now, for  $\,u =u(t,x)$ and $\,0\le d\le m-1$,
we consider the vector formed by the derivatives of order $\,d$ of $\,u$\,:
\begin{equation} \label{E-uji}
\uu{d}
  :\ug  \transpose{(u_{d,0},\dotsc,u_{d,d})} \in\, \R^{d+1},    \qquad  u_{d,j}:\ug  \partial_t^j \partial_x^{d-j} u \, ,
\qquad 0 \le j \le d.
\end{equation}
i.e.,
\begin{equation} \label{E-uj}
\uu{0}
 : \ug  u \, ,
\quad
\uu{1}
  :\ug  \begin{pmatrix}
      \partial_x u \\ \partial_t u
      \end{pmatrix} ,
\  \
\cdots \  \ , \  \
\uu{m-1}
  : \ug  \begin{pmatrix}
      \partial_x^{m-1} u \\
      \partial_t \partial_x^{m-2} u \\
      \vdots  \\
      \partial_t^{m-2} \partial_x u \\
      \partial_t^{m-1} u
      \end{pmatrix} \, .
\end{equation}
\sm
By a simple calculation
we~see that, if $\,u(t,x)$ is a solution  of \eqref{eq:scalar-eq},
then $\,\uu{d}$ satisfies the~system
\begin{align}
\partial_t \uu{d}
&  \ug  A_{d+1}(x) \, \partial_x \uu{d} \+ \alpha_{d+1}(x) \, \uu{d+1}\,,
  \qquad    \text{if} \quad d \ug 0,\dotsc, m-2, \label{E-S1} \\
\partial_t \uu{m-1}
&  \ug  A_m(x) \, \partial_x \uu{m-1}
     \+  \sum_{j=0}^{m-1} \rho_j(t,x) \, \uu{j}
     \+  \fff(t,x) \, , \label{E-S2}
\end{align}
where $\,\fff = \transpose{(0,\dotsc, 0, f)} \in\, \R^{d+1}$.
Next, setting
\begin{equation}
 \mu :\ug \frac{m(m-1)}2 \,,
\end{equation}
we~define the $\,\mu$-vectors
\begin{equation} \label{E-sU}
U   : \ug  \begin{pmatrix}
      \uu{0} \\
      \uu{1} \\
      \vdots  \\
      \uu{m-1}
      \end{pmatrix}  ,
\qquad
F
  :\ug  \begin{pmatrix}
      \,0\, \\ \,0\, \\
      \vdots  \\
      \,\fff \,
      \end{pmatrix}  .
\end{equation}
\gb
\noi Combining~\eqref{E-S1} and~\eqref{E-S2},
we conclude that the equation~\eqref{eq:scalar-eq} is~equivalent to~the~system
\begin{equation}
 \label{eq:sistema-storto}
U_t
  \ug  \sA(x) \, U_x
      \+ \sB(t,x) \, U
      \+ F(t,x) \, ,
\end{equation}
where $\,\sA$ and $\, \sB$
are $\mu \times \mu$ matrix
having a block structure
with blocks of~different size:
\begin{align}
   \label{eq:A-storto}  \sA(x)
&  :\ug  \begin{pmatrix}
       A_1(x) & 0 & \cdots & 0  \\
       \vdots & A_2(x) & & \vdots  \\
        & & \ddots &   \\
       0 & \cdots & 0 & A_m(x)
       \end{pmatrix} \, ,  \\
\sB(t,x)
&  : \ug \begin{pmatrix}
       0 & \alpha_1(x) & 0 & \cdots & 0  \\
       \vdots & & \alpha_2(x) & & \vdots  \\
       & & & \ddots & 0  \\
       0 & & \cdots & 0 & \alpha_{m-1}(x)  \\
       \rho_0(t,x) & \rho_1(t,x) & \cdots & & \rho_{m-1}(t,x)
       \end{pmatrix} \, .  \label{eq:B-storto}
\end{align}
\sm

\begin{Example} \
If $\,m=3$, we~have:
\begin{multline*}
\def\zero#1{\multicolumn{1}{c}{#1}}
\partial_t
\begin{pmatrix}
u  \\  u_x  \\  u_t  \\  u_{xx}  \\  u_{tx}  \\  u_{tt}
\end{pmatrix}
  \ug
\left(\begin{tabular}{c|cc|ccc}
$-a_{2,1}$  &  0  &  \zero{0}  &  0  &  0  &  0  \\\cline{1-3}
0  &  0  &  1  &  0  &  0  &  0  \\
0  &  $-a_{1,2}$  &  $-a_{1,1}$  &  0  &  0  &  0  \\\cline{2-6}
\zero{0}  &  0  &  0  &  0  &  1  &  0  \\
\zero{0}  &  0  &  0  &  0  &  0  &  1  \\
\zero{0}  &  0  &  0  &  $-a_{0,3}$  &  $-a_{0,2}$  &  $-a_{0,1}$
\end{tabular}\right)
\partial_x
\begin{pmatrix}
u  \\  u_x  \\  u_t  \\  u_{xx}  \\  u_{tx}  \\  u_{tt}
\end{pmatrix}  \\
  \+
\left(\begin{tabular}{c|cc|ccc}
0  &  $a_{2,1}$  &  1  &  0  &  0  &  0  \\\hline
0  &  0  &  0  &  0  &  0  &  0  \\
0  &  0  &  0  &  $a_{1,2}$  &  $a_{1,1}$  &  1  \\\hline
0  &  0  &  0  &  0  &  0  &  0  \\
0  &  0  &  0  &  0  &  0  &  0  \\
$r_{0,0}$  &  $r_{0,1}$  &  $r_{1,0}$
  &  $r_{0,2}$  &  $r_{1,1}$  &  $r_{2,0}$
\end{tabular}\right)
\begin{pmatrix}
u  \\  u_x  \\  u_t  \\  u_{xx}  \\  u_{tx}  \\  u_{tt}
\end{pmatrix} \, .
\end{multline*}
\end{Example}
\md

In order to construct a good symmetrizer for the matrix $\,\sA(x)$, we prove that each block $\,A_d(x)$  is the Sylvester matrix coming from a hyperbolic polynomial
which satisfies~\eqref{E-CO}, and we take the corresponding Jannelli symmetrizer.
Now, $\,A_d(x)$  is the Sylvester matrix of the  polynomial $\,\PP^{(d)}(x\pv\tau)$  (cf. \eqref{eq:monic}), which is proportional to the derivative $\,\de_\tau^{m-d}\PP(x\pv\tau)$; thus, we have  only to prove the following result:

\begin{Proposition} \label{P-der P}
Let~$\,\PP(x \pv \tau)$ be~a hyperbolic polynomial
with roots satisfying~\eqref{E-CO}
for some constant $\,M$.
Then  $\,\partial_\tau \PP(x \pv \tau)$ is~a hyperbolic polynomial
and its roots satisfy~\eqref{E-CO} with a constant~$\,\wt M\eqv \wt{M}(m, M)$  independent of~$\,x$.
Consequently, the polynomials $\,\PP^{(d)}(x\pv\tau)$  enjoy the same property.
\end{Proposition}

\begin{Remark} \label{R-der P} \
Before proving Proposition~\ref{P-der P}
we note that, if~$p(\tau)$ is a hyperbolic polynomial, then
 $p'(\tau)$ is also hyperbolic,
and the roots of~the two polynomials are~\emph{interlaced},
that is,
denoting by~$\,\tau_1\le \tau_2\le \dotsb\le \tau_m\,$  the roots of~$\,p(\tau)$,
and by $\,\lambda_1\le\lambda_2\le\dotsb\le\lambda_{m-1}$ those of~$\,p'(\tau)$,
one has
\begin{equation} \label{E-interlaced1}
\tau_j
  \mn  \lambda_j
  \mn  \tau_{j+1} \, ,
\qquad
j\ug 1,\dotsc,m-1.
\end{equation}
Indeed,~\eqref{E-interlaced1} follows
 from the classical Rolle's theorem, taking into account  that any root of~$\,p(\tau)$ with multiplicity $\,m>1$
is  a root of~$\,p'(\tau)$ with multiplicity $\,m-1$.
\end{Remark}
\smallskip

Let us now  recall that, in~\cite{Peyser-1967}, Peyser refined the estimate~\eqref{E-interlaced1} by
proving that
\begin{equation} \label{E-Peyser}
\tau_j \+ \frac{\tau_{j+1}-\tau_j}{m-j+1}
  \mn  \lambda_j
  \mn  \tau_{j+1}\, -\, \frac{\tau_{j+1}-\tau_j}{j+1} \, ,
\qquad
j=1,\dotsc,m-1 .
\end{equation}
Thus,~Proposition~\ref{P-der P} is  a direct consequence of the following Lemma.

\begin{Lemma} \label{L-pq}
Let $\,p(\tau)$ and~$\,r(\tau)$
be hyperbolic polynomials,
of degree~$\,m$ and~$\,m-1$ respectively,
whose roots,
$\,\tau_1\le\tau_2\le \dotsb\le \tau_m$
and
$\,\lambda_{1}\le\lambda_2\le \dotsb\le \lambda_{m-1}$,
satisfy
\begin{equation} \label{E-Nuij-ext}
\tau_j
  \mn  \lambda_j
  \mn  \tau_{j+1} \,-\, \eta \, (\tau_{j+1}-\tau_j) \, ,
\qquad
1\mn j\le m-1 \, ,
\end{equation}
for some constant~$\,\eta\in\left]0,1\right]$. Therefore, if
\begin{equation}
 \label{E-COs}
\tau_j^2 \+\tau_k^2
  \mn  M\,(\tau_j-\tau_k)^2,
\qquad
1\mn j<k \mn  m,
\end{equation}
for some constant $\,M$, then
\begin{equation} \label{E-COtilde}
\lambda_j^2 \+\lambda_k^2
  \mn  \wt{M}\,(\lambda_j-\lambda_k)^2,
\qquad
1\mn   j<k\le m \, ,
\end{equation}
with
\begin{equation} \label{c-tilde}
\wt{M} : \ug \frac{4M}{\eta^2}\+2 \, .
\end{equation}
\end{Lemma}

\begin{proof}
It is~sufficient to~prove~\eqref{E-COtilde} for $\,k=j+1$;
indeed, once proved~\eqref{E-COtilde} for $k=j+1$,
the~general case can be reached as follows:
\begin{align*}
\lambda_j^2+\lambda_{k}^2
&  \mn  \sum_{h=j}^{k-1}\, (\lambda_{h}^2+\lambda_{h+1}^2)
   \mn  \sum_{h=j}^{k-1}\, \wt{M}\,
          (\lambda_{h}-\lambda_{h+1})^2  \\
&  \mn  \wt{M}\,
        \bigg\{\sum_{h=j}^{k-1}
              \,(\lambda_{h}-\lambda_{h+1})
        \bigg\}^2
    \ug   \wt{M}\, (\lambda_j-\lambda_k)^2 \, .
\end{align*}

{}From the second inequality in~\eqref{E-Nuij-ext} it follows that
\begin{align*}
\tau_{j+1}\,-\,\tau_j
&  \mn  \frac{1}{\eta} \, (\tau_{j+1}\,-\,\lambda_j) \, ,
\intertext{hence}
\tau_{j+1}\,-\,\tau_j
&  \mn  \frac{1}{\eta} \, (\lambda_{j+1}\,-\,\lambda_j) \, .
\end{align*}
Then, noting that $\,\tau_j \le \lambda_j \le \tau_{j+1}$,
 and using \eqref{E-COs}, we get:
\begin{align*}
\lambda_j^2
&  \mn  \tau_j^2 \+ \tau_{j+1}^2
   \mn  	M\, (\tau_{j+1}-\tau_j)^2
   \mn  \frac{M}{\eta^2} \, (\lambda_{j+1}-\lambda_j)^2 \, .
\intertext{Analogously,
noting that $\,\tau_{j+1} \le \lambda_{j+1} \le \tau_{j+1} + (\lambda_{j+1}-\lambda_j)$,
we~get}
\lambda_{j+1}^2
&  \mn  \tau_{j+1}^2\+\bigl\{\tau_{j+1} \+ (\lambda_{j+1}\,-\,\lambda_j)\bigr\}^2
   \mn  3\, \tau_{j+1}^2 \+2\,(\lambda_{j+1}\,-\,\lambda_j)^2  \\
&  \mn  3\,M\,(\tau_{j+1}-\tau_j)^2 +2\,(\lambda_{j+1}\,-\,\lambda_j)^2
   \mn  \bigg\{\frac{3M}{\eta^2} \+ 2\bigg\} \, (\lambda_{j+1}\,-\,\lambda_j)^2 .
\end{align*}
Summing up, we~obtain~\eqref{E-COtilde}
with~$\wt{M}$ given by~\eqref{c-tilde}. This concludes the proof of Lemma \ref{L-pq}, and hence of Proposition \ref{P-der P}.
\end{proof}
\sm

\begin{Remark}
If, for some~$\,\eta\in\left]0,1\right]$, we replace~\eqref{E-Nuij-ext} by
\[
\tau_j \+ \eta \, (\tau_{j+1}-\tau_j)
  \mn  \lambda_j
  \mn  \tau_{j+1},
\qquad
1\mn j\mn m-1,
\]
 we~get the~same conclusion.
\end{Remark}
\sm

\begin{Remark}
{}From \eqref{E-interlaced1} it follows that
\[
\bigl|\taud{\ell}(x)\bigr| \mn \taumax \,,  \qquad \text{for} \ 1\le \ell\le d\le m.
\]
\end{Remark}
\gb

We are now in the position to construct the wished symmetrizer for~$\sA(x)$.
Indeed, the matrix $\,A_d(x)$ has the same properties of the matrix $\,A(x)$ in \S2,
thus it admits a good symmetrizer $\,Q_d(x)$
with the same properties of $\, Q(x)$ listed in Proposition~\ref{P-Q}.
\sm

Consequently, the $\,\mu\times \mu$ matrix
\begin{equation} \label{eq:Q-storto}
\sQ(x)
  : \ug  \begin{pmatrix}
        Q_1(x) & 0 & \cdots & & 0  \\
        \vdots & Q_2(x) & & & \vdots  \\
          & & \ddots &   \\
        \vdots & & & Q_{m-1}(x) & 0  \\
        0 & 0 & \cdots & 0 & Q_m(x)
        \end{pmatrix}
\end{equation}
is~a symmetrizer of~$\sA(x)$.
Note that, if $\,U$ is given by~\eqref{E-sU}, then:
\begin{equation} \label{E-sQuu}
\bigl(\sQ(x)\,U,U\bigr)
  \ug  \sum_{d=0}^{m-1} \,\bigl(Q_{d+1}(x)\, \uu{d},\uu{d}\bigr) \, .
\end{equation}

\noi Next we define the block diagonal matrix
\textup(with blocks of~different size\textup)
\begin{equation} \label{E-Xi}
\Xi(x)      
  : \ug  \begin{pmatrix}
      \Psi_1(x) & 0 & \cdots & 0  \\
       \vdots & \Psi_2(x) & & \vdots  \\
        & & \ddots &   \\
       0 & \cdots & 0 & \Psi_m(x)
       \end{pmatrix} \,,
    \end{equation}
where $\,\Psi_d(x)$ is~the $\,d\times d$ matrix
analogous to~$\,\Psi(x)$ in~\eqref{eq:Lambda-Psi}
related to~the~polynomial $\,\PP^{(d)}(x\pv\tau)$  (see~\eqref{eq:monic}).
 Then,  writing $\,\sA\equiv\sA(x), \,\sQ\equiv \sQ(x), \ldots$,
we  prove:
\sm

\begin{Proposition} \label{P-QQ}
The matrix $\,\sQ(x)$ in \eqref{eq:Q-storto} is a symmetrizer of~$\,\sA(x)$.
 If the~coefficients $\,a_j(x)$ of  the principal symbol \eqref{eq:princ-symb} verifiy \eqref{E-adk},
 there exists a constant
$\Gamma_0 \eqv \Gamma_0(m,L)$ such that,
for all $\,V \in \R^\mu$, it holds
\begin{align}
 \label{E-QQbdd}
\bigl|\bigl(\sQ \,V , V\bigr)\bigr|
&  \mn  \Gamma_0 \, |V|^2 ,  \\
\bigl|\bigl(\sQ\sA V , V\bigr)\bigr|
&  \mn  \taumax \,
       \bigl(\sQ \,V,V\bigr) \, , \label{E-QQA}  \\
\bigl|\bigl((\sQ\sA)_x V , V\bigr)\bigr|
&  \mn  \Gamma_0 \, \bigl(\Xi \,V , V\bigr) \, , \label{E-QQAp}  \\
 \label{E-AAp}
\bigl|\bigl(\sQ\sA_x\,V , \sA_xV\bigr)\bigr|
&   \eqv  \bigl|\sA_x V\bigr|^2
   \mn  \Gamma_0 \, \bigl(\Xi\, V , V\bigr) .
\end{align}
Moreover, if the operator $\,L$ in \eqref{eq:scalar-eq}
verifies Hypothesis~\ref{H-CO},
 there are two positive constants
$\,\Gamma_j$, $j=1,2$, such that:
\begin{equation} \label{E-cQ sim Xi-x}
\Gamma_1 \, \bigl(\Xi V,V\bigr)
   \mn  \bigl(\sQ V,V\bigr)
   \mn  \Gamma_2 \, \bigl(\Xi V,V\bigr) \, .
\end{equation}

\noi Finally, if
$\, \LL$ verifies also Hypothesis~\ref{H-ph},
there exists a~constant~$\,\Gamma_3$ such that:
\begin{equation} \label{E-QQB}
\Bigl|\bigl(\sQ(x)\sB(t,x)\,V,V\bigr)\Bigr|
  \mn  \Gamma_3 \, \bigl(\sQ(x) V,V\bigr) \, ,
\qquad
\text{for all} \ \,V\in\R^\mu.
\end{equation}
\end{Proposition} \sm

\begin{proof} \ The first part of proof  
is an immediate consequence of~Propositions~\ref{P-Q} and~\ref{P-der P},
so we~omit it, and we prove  only \eqref{E-QQB}:
\sm \gb

\noi By Schwarz inequality and \eqref{E-QQbdd},
we can find a constant~$\,C_1$ such that
\begin{equation} \label{E-33}
\Bigl|\bigl(\sQ\sB\,V,V\bigr)\Bigr|
  \mn  \bigl(\sQ\sB\, V,\sB V\bigr)^{1/2}
        (\sQ V,V)^{1/2}    \mn  C_1\,\bigl|\sB V\bigr|\,
          (\sQ V,V)^{1/2} ,
\end{equation}
for all $\,V \eqv \transpose{(\vvv_0,\dotsc, \vvv_{m-1})}\in \R^{\mu}$,
where $\vvv_{d} \in \R^{d+1}.$
Therefore, by~\eqref{eq:B-storto} we  have:
\[
V   : \ug  \begin{pmatrix}
      \vvv_{0} \\ \vvv_1 \\
      \vdots  \\
      \vvv_{m-2} \\
   \vvv_{m-1}
         \end{pmatrix}  , \qquad
  \sB \, V
  \ug  \begin{pmatrix}
     \alpha_1 \, \vv{1}  \\
     \vdots  \\
     \alpha_{m-1} \, \vv{m-1}  \\
     \sum_{d=0}^{m-1} \rho_d \, \vv{d}
     \end{pmatrix} \, ,
\]
with
\[
\alpha_d \, \vv{d}
  \ug  \begin{pmatrix}
     0  \\
     \vdots  \\
     0  \\
     \mathsf{vec}(\PP^{(d)}) \cdot \vv{d}
     \end{pmatrix} \, ,
\qquad
\rho_d \, \vv{d}
  \ug  \begin{pmatrix}
     0  \\
     \vdots  \\
     0  \\
     \mathsf{vec}(\RR_d) \cdot \vv{d}
     \end{pmatrix} \, ,
\]
and hence
\[
|\sB \, V|
  \mn   \sum_{d=1}^{m-1}\ \bigl| \mathsf{vec}(\PP^{(d)}) \cdot \vv{d}\,\bigr|
        \+  \sum_{d=0}^{m-1}\, \bigl|\mathsf{vec}(\RR_d) \cdot \vv{d}\,\bigr| \, .
\]
Now, according to Example~\ref{Ex-dtPP},
$\,\PP^{(d)}$ has a proper decomposition w.r.t.~$\,\PP^{(d+1)}$,
whereas, by~Hypotesis~\ref{H-ph}, 
$\,\RR_d$ has a proper decomposition w.r.t.~$\PP^{(d+1)}$.
Hence, thanks to Lemma~\ref{E-equi} and~\eqref{E-sQuu},
we have
\begin{equation} \label{E-34}
|\sB \, V|
  	\mn  C_2 \sum_{d=0}^{m-1}\,(Q_{d+1} \vv{d} , \vv{d})^{1/2}
  \mn  C_2 \,(\sQ V,V)^{1/2} \, ,
\end{equation}
for some constant $C_2$, and
inserting the inequality~\eqref{E-34} in~\eqref{E-33}
we~obtain~\eqref{E-QQB}.
\end{proof}
\md \gb

\section{The energy estimate}
\sm

Thanks to the (Propositions \ref{P-der P} and \ref{P-QQ}),
we are in the position to prove an energy estimate for the solutions $\,U\eqv U(t,x)$  to  the system
\begin{equation} \label{E-eqvG}
U_t
  \ug  \sA(x)\,U_x \+ \sB(t,x)\,U \+ F(t,x) .
\end{equation}
Indeed,   defining the energy
\[
 \sE(t,U)
  :\ug  \frac{1}{2} \int_{I_t}
      \bigl(\sQ(x) U(t,x),U(t,x)\bigr)\,dx,
\]
where $\,\sQ(x)$ is the symmetrizer of $\,\sA(x)$ in \eqref{eq:Q-storto},  $\,I_t$ the real interval~\eqref{E-def It}, we prove: \sm

\begin{Proposition} \label{P-d=m}
Assume that  the operator $\,\LL$ in \eqref{eq:scalar-eq} satisfies 
the~Hypothesis~\ref{H-H}, \ref{H-CO} and \ref{H-ph}, and  the coefficients  $\,a_j(x)$ belong to $\, \mathcal{C}^\infty \cap L^\infty$. Then, 
 for each solution~$U$ to \eqref{E-eqvG}, 
  it~holds
the estimate:
\begin{equation} \label{E-systemg}
\sE(t,U)
  \mn  C\,\biggl\{ \sE(0,U)
       \+  \int_0^t
          \bigl\|F(s,\cdot)\bigr\|_{L^2(I_s)}^2 \, ds\biggr\} \, ,
\qquad
0\le t\le T_0 \, ,
\end{equation}
with
$\,T_0 := \min\,\bigl\{\rho_0/\taumax,\, T\bigr\}$
and $\,C$  independent of~$\,t$ and $x$.
\end{Proposition}
\gb

\begin{proof}
The proof is similar to~that
of Lemma~6.1 in~\cite{ST-JHDE}.
Differentiating~$\sE(t)$, we~find:
\[
\sE'(t,U) \ug  \int_{I_t}\bigl(\sQ(x) U_t, U\bigr)\,dx \+ I_4(t)
  \ug  I_1(t) \+ I_2(t) \+ I_3(t) \+ I_4(t) \, ,
\]
where
\begin{gather*}
I_1(t)
  :\ug  \int_{I_t}(\sQ \sA\, U_x , U)dx  ,
\quad
I_2(t)
  :\ug  \int_{I_t}(\sQ \sB \,U, U)dx ,
\quad
I_3(t)
  :\ug  \int_{I_t}(\sQ F, U)dx  ,  
  \end{gather*}
  and
\[
I_4(t)
  :\ug  - \frac{\taumax}{2}\,
        \bigg\{(\sQ U, U)\bigm|_{x=x_0+\rho(t)}
               + (\sQ U, U)\bigm|_{x=x_0-\rho(t)}\bigg\} \, .
\]
\gb
\noi Since the matrix~$\,\sQ(x)\sA(x)$ is~symmetric,
the identity
\[
\bigl((\sQ \sA)U,U\bigr)_x
  \ug  \bigl((\sQ \sA)_xU,U\bigr)
     \+  2 \bigl((\sQ \sA)U_x,U\bigr)
\]
holds true; hence, recalling~\eqref{E-QQA}, \eqref{E-QQAp}
and~\eqref{E-cQ sim Xi-x}, we derive
\begin{align*}
I_1(t)
&  \ug \, \frac{1}{2}  \, \bigg\{(\sQ \sA U, U)
                  \bigm|_{x=x_0+\rho(t)}
        -\,  (\sQ \sA U, U)\bigm|_{x=x_0-\rho(t)}\bigg\}
        \,-\,  \frac{1}{2}  \,
           \int_{I_t}\bigl((\sQ \sA)_xU,U\bigr)dx \\
&  \mn \,  \frac{\taumax}{2}\,
              \bigg\{(\sQ \, U, U)\bigm|_{x=x_0+\rho(t)}
               \+ (\sQ \, U, U)\bigm|_{x=x_0-\rho(t)}\bigg\}
               \+\,  \frac{\Gamma_0}{\Gamma_1} \,
                  \sE(t)  .
\end{align*}
Consequently:
\[
I_1(t) \+ I_4(t)
  \mn  \,\frac{\Gamma_0}{\Gamma_1} \, \sE(t) .
\]
On the~other hand, by~Schwarz' inequality and~\eqref{E-QQbdd},
we~have
\begin{align*}
I_2(t) &\mn \Gamma_3\, \sE(t),\\
I_3(t)
&  \mn  \Bigl\{\int_{I_t}(\sQ F , F)dx\,\Bigr\}^{1/2} \sqrt{2 \, \sE(t)\,}
  \mn  \frac{1}{2} \, \int_{I_t}(\sQ F , F)dx \+ \sE(t) \\
&   \mn  \frac{\Gamma_0}{2} \,
        \bigl\|F(t,\cdot)\bigr\|_{L^2(I_t)}^2
        \+ \sE(t) .
\end{align*}
Summing up, we~find a constant
$\,C=C(\Gamma_j,d_0)
  \equiv C(m,L,M,d_0)$ for which
  \[
\sE' (t,U)
  \mn  C \, \biggl\{ \sE(0,U)
       \+  \bigl\|F(t,\cdot)\bigr\|_{L^2(I_t)}^2 \biggr\} ,
\qquad
0\le t\le T_0 ,
\]
so, by~Gr\"{o}nwall Lemma, 
we~get~\eqref{E-systemg}.
\end{proof}
\sm

The next goal will be to derive from  \eqref{E-systemg} an estimate for the Sobolev norms of the solution $\,U(t,x)$. This is not easy because the symmetrizer $\,\sQ(x)$ is not coercive.
First of all we observe that 
\[
\sE(t,U)
  \mg  \frac{\Gamma_1}{2} \, \sum_{d=0}^{m-1} \,\int_{I_t}(\Psi_d(t) \,\uu{d} , \uu{d})\,dx,
\]
and hence, since each $\,\Psi_d$ is a diagonal matrix
with  $\,(d,d)$-entry  $1$,
\[
\sE(t,U)
  \mg  \frac{\Gamma_1}{2}\,\sum_{d=0}^{m-1}\, \int_{I_t}\bigr|u_{d,d}\bigl|^2dx
  \ug  \frac{\Gamma_1}{2}\, \sum_{d=0}^{m-1}\,\int_{I_t}\bigl|\partial_t^d u \bigr|^2dx  .
\]
Thus~\eqref{E-systemg} provides an estimate for the time derivatives
of~the solution $u$ to~\eqref{E-scalar}-\eqref{E-data}:
\begin{equation} \label{eq:time-derivatives}
\ \sum_{d=0}^{m-1}\,\bigl\|\partial_t^d u(t,\cdot)\bigr\|_{L^2(I_t)}^2
  \le  C\,\Bigl\{ \sum_{j=0}^{m-1} \bigl\|\partial_t^d u(t,0)\bigr\|_{L^2(I_0)}
       +  \int_0^t
          \bigl\|f(s,\cdot)\bigr\|_{L^2(I_s)}^2 ds\Bigr\} ,
\quad
0\le t\le T_0 \, .
\end{equation}
In \S 5 we derive   a similar estimate for the {\it space derivatives}, by proving that   $\,\de_x^ku$  is a solution of 
an equation, like~\eqref{E-scalar}, which  satisfies
 the Hypothesis~\ref{H-H}, \ref{H-CO} and \ref{H-ph}.
\sm \gb

\section{Estimate of the space-derivatives}

We have
\[
\partial_x \bigl\{\PP(x\pv\partial_t,\partial_x)u\bigr\}
  \ug  \PP(x\pv\partial_t,\partial_x)u_x
     \+ \PP'(x\pv\partial_t,\partial_x) u_x,
\]
where $\,\PP'(x\pv\partial_t,\partial_x)$ is the differential operator of~order $m-1$
with symbol
\begin{equation}\label{eq:P-primo}
\PP'(x\pv\tau,\xi)
  :\ug  \frac{1}{\xi} \, \partial_x \PP(x\pv \tau,\xi)
   \equiv\,   \sum_{j=1}^m a_j'(x) \tau^{m-j}\,\xi^{j-1} .
\end{equation}
\gb
\noi Similarly,  recalling \eqref{eq:lot}, we find:
\[
\partial_x \bigl\{\RR_d(t,x \pv\partial_t,\partial_x)u\bigr\}
  \ug  \RR_d(t,x\pv\partial_t,\partial_x)u_x
     \+ \partial_x r_{d,0}(t,x) \de_t^d u
     \+ \RR_d'(t,x \pv\partial_t,\partial_x)u_x,
\]
where $\,\RR_d'(t,x\pv\partial_t,\partial_x)$ is the differential operator of~order $\,d-1$
with symbol
\begin{equation} \label{eq:Rd-prime}
\RR_d'(t,x\pv\tau,\xi)
  :\ug   \frac{1}{\xi} \, \partial_x \Bigl\{\RR_d(t,x \pv\tau,\xi) - r_{d,0}(t,x) \tau^d\Bigr\}
   \equiv  \sum_{j=1}^d \partial_x r_{d,j}(t,x) \tau^{d-j}\,\xi^{j-1} .
\end{equation}
Differentiating~\eqref{E-scalar} w.r.t.~$x$,
we see that $\,u_x$ satisfies the equation 
\begin{equation} \label{E-scalar1}
L^{(1)}(t,x\pv\partial_t,\partial_x)u_x
  \ug  f^{(1)}(t,x),
\end{equation}
where
\begin{equation} \label{E-L1}
L^{(1)}(t,x\pv\partial_t,\partial_x)
  :\ug   \PP(x\pv\partial_t,\partial_x)
  \,-\,  \sum_{d=0}^{m-1} \RR_d^{(1)}(t,x\pv \partial_t,\partial_x) \end{equation}
and
\[
\RR_d^{(1)}(t,x\pv\tau,\xi)
  :\ug  \begin{cases}
      \RR_{m-1}(t,x \pv\tau,\xi) \,-\, \PP'(x\pv\tau,\xi)
      &  \text{if~$d=m-1$} \, ,  \\
      \RR_d(t,x\pv\tau,\xi) \+ \RR_{d+1}'(t,x\pv\tau,\xi) \,
      &  \text{if~$d=1,\dotsc,m-2$} \, ,
\end{cases}
\]
\[
f^{(1)}(t,x)
  :\ug \partial_x f(t,x) \,-\, \sum_{d=0}^{m-1} \partial_x r_{d,0}(t,x) \partial_t^d u .
\]
Since the principal symbol of~\eqref{E-scalar1}
is the same of~that of~\eqref{E-scalar},
Hypothesis~\ref{H-H} and \ref{H-CO} for the operator $\,L^{(1)}(t,x\pv \de_t, \de_x)$ are fulfilled.
 Thus, we check that $\,L^{(1)}$ satisfies  Hypothesis~\ref{H-ph}, i.e.,
the polynomial~$\,\RR_d^{(1)}(t,x\pv\tau,\xi)$
has a proper decomposition w.r.t.\ $\,\partial_\tau^{m-d} \PP(x\pv\tau,\xi)$.
\sm

\noi Now, according to Example~\ref{Ex-dtPP},  the polynomial
$\,\PP'(x\pv\tau)$ in  \eqref{eq:P-primo} has a proper decomposition w.r.t. $\PP(x\pv\tau)$, while $\,\RR_d(t,x\pv \de_t, \de_x)$ has a proper decomposition w.r.t. $\,\PP^{(d)}(x\pv\tau)$, 
hence it remains only to prove that
\begin{equation} \label{eq:remains}
\RR_{d+1}'(t,x\pv\tau)
\ \  \text{\it has a proper decomposition w.r.t.} \ \PP^{(d)}(x\pv\tau), \quad d=1, \dotsc, m-1.
\end{equation}
More precisely, we prove:
\begin{Proposition} \label{P-de k b}
Let $\,\PP(x \pv\tau,\xi)$ be a~hyperbolic polynomial in $(\tau,\xi)$ of~degree~$\,m$
with characteristic roots verifying~\eqref{E-CO},
and~$\,\RR(t,x \pv\tau,\xi)$  an homogeneous polynomial in $\,(\tau,\xi)$ of~degree~$\,m-1$ having a proper decomposition w.r.t.~$\,\PP$.
Then, the polynomial
\[
\RR'(t,x\pv\tau,\xi)
  :\ug  \frac{1}{\xi} \,
      \partial_x \Bigl\{\RR(t,x\pv\tau,\xi)
              \,-\,r_0(t,x) \tau^{m-1}\Bigr\} ,
\]
where $\,r_0$ is the leading coefficient of~$\,\RR$,
 has a proper decomposition w.r.t. $\,\partial_\tau \PP$.
 Consequently \eqref{eq:remains} holds true and the operator $\, L^{(1)}$ in \eqref{E-L1} satisfies also
  Hypothesis \ref{H-ph}.
\end{Proposition}

\noi Consequently, we can apply Proposition \ref{P-d=m} to the equation \eqref{E-scalar1} to estimate  the $x$-derivative of the solutions $\,u$ of \eqref{eq:scalar-eq}. 
Summing up we  obtain: 
\begin{Corollary}
\noi If the operator $\,L(t,x\pv \de_t, \de_x)$ in  equation \eqref{eq:scalar-eq} verifies the Hypothesis~\ref{H-H}, \ref{H-CO} and~\ref{H-ph}, then also  $\,L^{(1)}(t,x\pv \de_t, \de_x)$ in \eqref{E-L1} verifies these Hypothesis. \\
As a consequence  every smooth solution $\,u(t,x)$ of \eqref{eq:scalar-eq} satisfies an apriori estimate like \eqref{eq:time-derivatives} with the $\,H^1$-\,norm in place of the $\,L^2$-\,norm.
Iterating, we derive an apriori estimate for any $\,H^k$-\,norm of $\,u$. 
\end{Corollary}

\noi The proof of  Proposition \ref{P-de k b}
is not very direct and needs some preliminary Lemmas.
We start by fixing notation and terminology.
To simplify the presentation, throughout this section 
we~omit the~$t$ and $\xi$ variables, and we  consider the polynomials:
\[
 \PP(x\pv\tau)
   \ug  \sum_{h=0}^m a_h(x) \, \tau^{m-h}
   \equiv  \prod_{j=1}^m \bigl(\tau-\tau_j(x)\bigr) \, ,
\]
where $\,\{\tau_1(x)\le\dotsb\le\tau_m(x)\}$ are real valued functions, and
\[
 \RR(x\pv\tau)
   \ug  \sum_{h=0}^{m-1} r_h(x) \, \tau^{m-1-h}
   \equiv  \prod_{j=1}^{m-1} \bigl(\tau-\lambda_j(x)\bigr) \, .
\]
The  $\lambda_j$'s are not necessarly real valued functions.
\sm

We recall that~$\,\RR$
has a~\emph{proper decomposition \textup(of first order\textup) w.r.t.~$\,\PP$}
if
\begin{equation} \label{E-dec 1}
\RR(x;\tau)
  \ug  \sum_{h=1}^m \, \ell_h(x) \, \PP_{\wh{h}}(x\pv\tau)  
\ \ \quad
\text{for some $\,\ell_h\in L^\infty(\R)$} \, ,
\end{equation}
where the $\PP_{\wh{h}}$ are
the~\emph{reduced polynomials} of~$\,\PP$ defined in~\eqref{eq:reduced}.

\noi Analogously,
we say that a polynomial $\,\RR$ of~degree~$\,\le m-2\,$
has a~\emph{proper decomposition of~second order w.r.t.~$\,\PP$}
if
\[
\RR(x;\tau)
  \ug  \sum_{1\le h<k\le m} \ell_{h,k}(x) \, \PP_{\wh{h},\wh{k}}(x;\tau)
\ \ \quad
\text{for some $\, \ell_{h,k}\in L^\infty(\R)$},
\]
where the $\PP_{\wh{h},\wh{k}}$ are
the~\emph{bi-reduced polynomials of~$\,\PP$}, that is:
\begin{equation} \label{E-bired}
\PP_{\wh{h},\wh{k}}(x \pv\tau)
  :\ug  \prod_{\substack{\ell=1,\dotsc,m \\ \ell\notin\{h,k\}}}
      \bigl(\tau-\tau_\ell(x)\bigr) \, .
\end{equation}
If also the roots $\,\lambda_j$  of~$\,\RR$ are real,
we say that the polynomials $\,\PP$ and $\,\RR\,$ are \emph{interlaced}
if
\begin{equation} \label{E-inter2}
\tau_j(x)
  \mn  \lambda_j(x)
  \mn  \tau_{j+1}(x),\qquad \text{for all} \ x\in\R, \  j=1,\dotsc,m-1.
\end{equation}

Now we give a series of Lemmas
that will be used to prove Proposition~\ref{P-de k b}.

\begin{Lemma} \label{L-tau dec}
If~$\,\PP(x \pv\tau)$ verifies Hypothesis~\ref{H-CO},
then $\,\tau^{m-1}$ has a~proper decomposition w.r.t.~$\,\PP(x\pv\tau)$.
\end{Lemma}

\begin{proof}
As $\,\mathsf{vec}(\tau^{m-1}) = (0,\dotsc,0,1)$,
we have
\[
\bigl(\,\mathsf{vec}(\tau^{m-1}) \,,\, v\bigr) \ug v_m \, ,
\qquad
\text{for any $v=(v_1,\dotsc,v_m)^T\in\R^m$} \, .
\]
Thus the proof follows from~\eqref{E-weakcoercivity} and Lemma~\ref{E-equi}.
\end{proof}

In the next Lemma we consider the relation between
proper decomposition and  interlacing property.
\gb

\begin{Lemma} \label{L-Fisk} \hfill
\begin{enumerate}
\item  Let the leading term of~$\,\RR$ be positive.
       Then~$\,\PP$ and $\,\RR$ are interlaced
       if and only if
       $\,\RR$ has a~proper decomposition w.r.t.~$\,\PP$
       with coefficients $\,\ell_h(x)$ in~\eqref{E-dec 1}  positive.
       \sm

\item  Let $\,\RR$ have a~proper decomposition w.r.t.~$\,\PP$.
       Then $\,\RR$ is the sum of~two hyperbolic polynomials of~degree~$m-1$
       both interlaced with~$\,\PP$.
More precisely,
       there exists a constant $\,\zeta>0$ and a monic hyperbolic polynomial $\,\wt{\PP}_\zeta(x;\tau)$
       interlaced with $\,\PP(x;\tau)$, such that
              \begin{equation} \label{E-RPt}
       \RR(x;\tau)
         \ug  \zeta\, \partial_\tau\PP(x;\tau)
              \+ \bigl(r_0(x) - m \zeta\bigr) \, \wt{\PP}_\zeta(x;\tau) \, .
       \end{equation}
\end{enumerate}
\end{Lemma}

\begin{proof}
The first part of the Lemma  is already known:
see Lem\-ma~1.20 in~\cite{Fisk}.
\sm

\noi We prove the second part. Firtsly recall that the polynomial $\,\PP_\tau$ is interlaced with $\,\PP$ (Remark \ref{R-der P}).
On the other side, according to~\eqref{E-dec 1}, and~\eqref{E-Pp Ph},
\[
\zeta\, \partial_\tau\PP(x;\tau) \,-\, \RR(x;\tau)
  \ug  \sum_{h=1}^m \, \bigl(\zeta - \ell_h(x)\bigr) \, \PP_{\wh{h}}(x;\tau) .
\]
Thus, choosing $\,\zeta\,$ large enough,
the coefficients $\,\zeta - \ell_h(x)$ are all positive, 
and hence, by the part (1), the~polynomial
\[
\wt{\PP}_\zeta(x;\tau)
  :\ug  \frac{1}{m \zeta - r_0(x)} \, \Bigl\{\zeta\, \partial_\tau\PP(x;\tau) - \RR(x;\tau)\Bigr\}
\]
is  monic, hyperbolic and
 interlaced with~$\PP(x;\tau)$.
\end{proof}
\gb

Here we give some condition
in order to have a~proper decomposition of second order.

\begin{Lemma} \label{L-inter2}
Let $\,\SS(x\pv\tau)$ be a polynomial of~degree~$\,\le m-2$,
and assume that one of the following conditions is fulfilled\textup:
\begin{enumerate}
\item  $\SS$ has a~proper decomposition w.r.t.~$\,\PP$;
\sm
\item  there exists a polynomial $\,\RR$ of degree $\,\le m-1$
       interlaced with $\,\PP$ such that 
       $\SS$ has a~proper decomposition w.r.t.~$\,\RR$.
\end{enumerate}

\noi Then $\,\SS(x\pv\tau)$ has a~proper decomposition of second order w.r.t.~$\,\PP$.
\end{Lemma}

\begin{proof}
(1) \
Since $\,\SS(x \pv\tau)$ is of~degree~$\,\le m-2$,
by considering the terms of~degree $\,m-1$ in \eqref{E-dec 1}, we see that
\[
\sum_{h=1}^m \, \ell_h(x)  \ug 0 ,
\]
thus:
\begin{align*}
\SS(x\pv\tau)
&  \ug  \sum_{h=1}^m \, \ell_h(x) \, \PP_{\wh{h}}(x\pv\tau)
     \,-\, \Bigl\{\sum_{h=1}^m \, \ell_h(x)\Bigr\} \, \PP_{\wh{1}}(x\pv\tau)  \\
&  \ug  \sum_{h=2}^m \, \ell_h(x) \, \Bigl\{\PP_{\wh{h}}(x\pv\tau) - \PP_{\wh{1}}(x\pv\tau)\Bigr\} .
\end{align*}

\noi To derive the wished conclusion, we have only to observe that, for any $h\ne k$, it holds
\[
\PP_{\wh{k}}(x;\tau)-\PP_{\wh{h}}(x \pv  \tau)
  \ug  (\tau-\tau_k) \, \PP_{\wh{h},\wh{k}}(x\pv\tau) - (\tau-\tau_h) \, \PP_{\wh{h},\wh{k}}(x \pv \tau)
  \ug  (\tau_h-\tau_k) \, \PP_{\wh{h},\wh{k}}(x\pv\tau) .
\]

\noi (2) \,
By linearity, we can assume,
with no loss of~generality,
that $\,\SS(x;\tau) = \RR_{\wh{\jmath}}(x;\tau)$, for some $j$.
As $\,\RR_{\wh{\jmath}}$ is interlaced with $\,\PP_{\wh{\jmath}}$,
thanks to Lemma~\ref{L-Fisk}-(1),
it has a proper decomposition w.r.t.~$\PP_{\wh{\jmath}}$,
hence it has a proper decomposition of second order w.r.t.~$\PP$.
\end{proof}

The converse of~Lemma~\ref{L-inter2} is false in general.

\begin{Example}
Let
\begin{align*}
\PP(x;\tau)
&  \ug  (\tau^2-x^2)\,(\tau^2-1)  \\
\SS(x\pv\tau)
&  \ug  (\tau^2-1) \, .
\end{align*}
$\,\SS(x\pv \tau)$ has a proper decomposition of~second order w.r.t.~$\,\PP(x\pv\tau)$,
but is not properly decomposable w.r.t.~$\,\PP(x\pv\tau)$.
Indeed, there's a unique way in which we can write $\SS(x\pv\tau)$
as a linear combination of the reduced polynomials of~$\,\PP(x\pv\tau)$
(cf.~Remark~\ref{R-1}),i.e., 
\[
\tau^2-1
  \ug  \frac{1}{2x} \, (\tau+x)\,(\tau^2-1)
     \,-\,  \frac{1}{2x} \, (\tau-x)\,(\tau^2-1),
\]
and the coefficients are not bouded.
On the other side,
if
\[
\RR(x\pv\tau)
  \ug  \tau\,(\tau^2-4\,x^2),
\]
$\,\PP$ and $\,\RR$ are interlaced polynomials,
but $\,\SS$ is not properly decomposable w.r.t.~$\,\RR$,
since any linear combination of~the reduced polynomials of~$\,\RR$
vanishes for $\,x=0$ and $\,\tau=0$.
\end{Example}
\sm
\gb

This example shows that,
in order to~reverse Lemma~\ref{L-inter2}-(2),
if~$\,\tau_j(x)<\lambda_j(x)<\tau_{j+1}(x)$
and $\,\lambda_j(x)$ approaches $\,\tau_j(x)$,
then $\,\tau_{j+1}(x)$ should approach them too.
Roughly speaking, $\,\lambda_j(x)$ has to~remain in some ``central part''
of the interval~$\,\bigl[\tau_j(x),\tau_{j+1}(x)\bigr]$.
\sm

Motivated by Peyser's inequality~\eqref{E-Peyser},
we~now consider the estimate
\begin{equation} \label{E-Peysergen}
\tau_j(x) \+ \eta\,\bigl(\tau_{j+1}(x)-\tau_j(x)\bigr)
  \mn  \lambda_j(x)
  \mn  \tau_{j+1}(x) \,-\, \eta\,\bigl(\tau_{j+1}(x)-\tau_j(x)\bigr) ,
\end{equation}
for some~$\,0<\eta<1/2$, and all $\,j=1,\dotsc,m-1$.
\sm

Note that from the interlacing property~\eqref{E-inter2}
it follows that the ratios
\begin{equation} \label{E-ratio1}
\frac{\tau_{j+1}(x)-\lambda_j(x)}{\lambda_{j+1}(x)-\lambda_j(x)}
\qquad\qquad
\frac{\lambda_{j+1}(x)-\tau_{j+1}(x)}{\lambda_{j+1}(x)-\lambda_j(x)}
\end{equation}
are bounded,
whereas from~\eqref{E-Peysergen}
it~follows
\[
\eta\,\bigl(\tau_{j+2}(x)-\tau_j(x)\bigr)
  \mn  \lambda_{j+1}(x)-\lambda_j(x)
  \mn  (1-\eta) \, \bigl(\tau_{j+2}(x)-\tau_j(x)\bigr) \, ,
\]
hence also the ratios
\begin{alignat}{2}
&  \frac{\lambda_j(x)-\tau_j(x)}{\lambda_{j+1}(x)-\lambda_j(x)}  & \qquad\qquad
&  \frac{\lambda_{j+1}(x)-\tau_j(x)}{\lambda_{j+1}(x)-\lambda_j(x)} \label{E-ratio2}  \\
&  \frac{\tau_{j+2}(x)-\lambda_j(x)}{\lambda_{j+1}(x)-\lambda_j(x)}  &
&  \frac{\tau_{j+2}(x)-\lambda_{j+1}(x)}{\lambda_{j+1}(x)-\lambda_j(x)} \label{E-ratio3}
\end{alignat}
are bounded.

\begin{Lemma}
Assume that~$\,\,\PP$ and $\RR$ are interlaced
and~their roots verify~\eqref{E-Peysergen}.
Then any polynomial $\,\SS$ having
a~proper decomposition of~second order w.r.t.~$\,\PP$
has a~proper decomposition w.r.t.~$\,\RR$.
\end{Lemma}

\begin{proof}
We give the proof in the case $m=3$.
The proof in the general case requires only additional technical effort.
\sm

\noi If $\,m=3$ we have
\begin{equation} \label{E-inter4}
\tau_1(x)
  \mn \lambda_1(x)
  \mn \tau_2(x)
  \mn \lambda_2(x)
  \mn \tau_3(x) \, ,
\end{equation}
and we have to prove that
\[
\tau-\tau_1(x) \, ,\, \quad
\tau-\tau_2(x) \, ,\, \quad
\tau-\tau_3(x)   
\]
are linear combination with bounded coefficients
of
\[
\tau-\lambda_1(x) \, ,\, \quad
\tau-\lambda_2(x) \, .
\]
Now, by Newton Lagrange formula, we have
\begin{align*}
\tau-\tau_1(x)
&  \ug  \frac{\lambda_2(x)-\tau_1(x)}{\lambda_2(x)-\lambda_1(x)} \, \bigl(\tau-\lambda_1(x)\bigr)
      \+  \frac{\lambda_1(x)-\tau_1(x)}{\lambda_1(x)-\lambda_2(x)} \, \bigl(\tau-\lambda_2(x)\bigr)  \\
\tau-\tau_2(x)
&  \ug  \frac{\lambda_2(x)-\tau_2(x)}{\lambda_2(x)-\lambda_1(x)} \, \bigl(\tau-\lambda_1(x)\bigr)
      \+  \frac{\lambda_1(x)-\tau_2(x)}{\lambda_1(x)-\lambda_2(x)} \, \bigl(\tau-\lambda_2(x)\bigr)  \\
\tau-\tau_3(x)
&  \ug  \frac{\lambda_2(x)-\tau_3(x)}{\lambda_2(x)-\lambda_1(x)} \, \bigl(\tau-\lambda_1(x)\bigr)
      \+  \frac{\lambda_1(x)-\tau_3(x)}{\lambda_1(x)-\lambda_2(x)} \, \bigl(\tau-\lambda_2(x)\bigr).
\end{align*}

Thanks to~\eqref{E-ratio1}, \eqref{E-ratio2} and~\eqref{E-ratio3}
all the coefficients are bounded.
\end{proof}

\sm

Finally, from Peyser's inequality~\cite{Peyser-1967} (cf.~\eqref{E-Peyser})
we derive the~following result:

\begin{Lemma} \label{L-interrev}
Any polynomial $\SS(x;\tau)$ having
a~proper decomposition of~second order w.r.t.~$\PP(x;\tau)$
has a~proper decomposition w.r.t.~$\partial_\tau \PP(x;\tau)$.
\end{Lemma}

Now we have all the tools necessary to prove  Proposition~\ref{P-de k b}.

\begin{proof}[\textbf{Proof of~Proposition~\ref{P-de k b}.}]
By~\eqref{E-RPt} we derive the identity:
\begin{equation} \label{E-last}
\begin{split}
\partial_x \Bigl\{\RR(x\pv\tau) - r_0(x) \tau^{m-1}\Bigr\}
&    \\ \ug \zeta\, \partial^2_{x\tau}\PP(x\pv\tau)
          & \+ \bigl(r_0(x) - m \zeta\bigr) \, \partial_x \wt{\PP}_\zeta(x\pv\tau)  \+ r_0'(x) \bigl( \wt{\PP}_\zeta(x\pv\tau) - \tau^{m-1}\bigr),
\end{split}
\end{equation}
and we prove that each of the  terms on the left  
has a proper decomposition w.r.t.~$\,\partial_\tau \PP$.
\sm

Since $\, \partial_\tau \PP(x\pv\tau)$ is a hyperbolic polynomial,
using~\eqref{E-Px Ph} with~$\,\partial_\tau \PP$ in place of $\,\PP$,
we~see that $\,\partial^2_{x\tau} \PP$
has a proper decomposition w.r.t.~$\,\partial_\tau \PP$.
\sm

For the second term,
as~$\,\wt{\PP}_\zeta$ is a monic hyperbolic polynomial,
using~\eqref{E-Px Ph} with~$\,\wt{\PP}_\zeta$ in place of~$\,\PP$
we~see that $\,\partial_x \wt{\PP}_\zeta$
has a proper decomposition w.r.t.~$\,\wt{\PP}_\zeta$.
As~the roots of~$\,\wt{\PP}_\zeta$ are interlaced with those of~$\,\PP$,
by Lemma~\ref{L-inter2}-(2),
$\,\partial_x \wt{\PP}_\zeta$ 
has a proper decomposition of second order w.r.t.~$\,\PP$, hence  by Lemma~\ref{L-interrev} it has   a proper decomposition  w.r.t.~$\partial_\tau \PP$. 
\sm

For the third term,
as~$\,\wt{\PP}_\zeta$ is interlaced with~$\,\PP$,
by Lemma~\ref{L-Fisk}-(1), it
 has a proper decomposition w.r.t.~$\PP$.
Moreover, by~Lemma~\ref{L-tau dec},
$\,\tau^{m-1}$ has a proper decomposition w.r.t.~$\,\PP$ too.
Thus $\,\wt{\PP}_\zeta(x\pv\tau)- \tau^{m-1}$
is a polynomial of~order $\le m-2\,$
which has  a proper decomposition w.r.t.~$\,\PP$, 
hence  also a proper decomposition of second order w.r.t.~$\,\PP$
and by Lemma~\ref{L-interrev} a proper decomposition  w.r.t.~$\partial_\tau \PP$. 
\end{proof}

\section {Nuij's approximation}

We have proved an apriori estimate for each space-time derivative of the solutions to  $\, \LL u = f$, where $\,\LL\eqv \PP-\bigl\{\RR_{m-1}+ \cdots +\RR_0\}$ is the operator \eqref{eq:L}.  
\sm

\noi To construct a solution for the corresponding Cauchy Problem, we proceed as in~\cite{ST-JHDE}: we approximate the operator $\,\LL$ by a sequence $\,\{\LL_\e\}$ of strictly hyperbolic operators,  
to get a sequence $\,\{u_\e\}$ of approximating  solutions  converging to some solution of \eqref{E-scalar}. 
\sm

\noi More precisely we approximate the principal symbol of $\,\LL$, i.e., the polynomial $\,\PP(x\pv \tau,\xi)$, by the polynomials
\begin{equation} \label{E-Nuij}
\PP_\e :\ug \sN_e^{m-1} (\PP),
\end{equation}
where $\,\sN_\e$ is the {\it Nuij's map}:
\begin{equation}
 \ \sN_\e : \ P  \ \longmapsto \  P -\e\,\de_\tau P.
\end{equation}

\noi In~\cite{Nuij}, Nuij proved that, if $\,P(\tau)$ is a hyperbolic polynomial with maximum multiplicity $\,r$,
then $P-\e\,\de_\tau P\,$ is a hyperbolic polynomial with maximum multiplicity $\,r-1$. Thus the polynomial $\,\sN_\e^{m-1}P$ is strictly hyperbolic. 
Then we define 
\[
\LL_\e := \PP_\e -\bigl\{\RR_{m-1}+ \cdots +\RR_0 \bigr\}.
\]

\noi Now, let $\,u_\e$ be the solution of
the Cauchy Problem \eqref{E-scalar}-\eqref{E-data}
with $\,\LL_\e$ in place of $\,\LL$.
\sm
 
Before applying our apriori estimate to  the functions $\,u_\e$, we check that the polynomials $\,\LL_\e(t,x\pv\tau,\xi)$ verify the Hypothesis~\ref{H-H},~\ref{H-CO} and~\ref{H-ph}, uniformly w.r.t.~$\e$.

Now $\,\LL_\e(t,x\pv\tau,\xi)$ verify Hypothesis~\ref{H-H} by Nuij's Lemma,
and Hypothesis~\ref{H-CO} by Lemma 1.9 in~\cite{ST-JHDE}.
To check Hypothesis~\ref{H-ph}
we use the following result of S.~Wakabayashi (see~\cite{Wakabayashi-1986})
which precises the Nuij's approximation:

\begin{Lemma} \label{L-waka}
Let $\,P(\tau)$ an hyperbolic polynomial with degree $\,m$, and $\,P_\e 
= \sN_\e (P)$.  Denote by $\{\tau_j\}$ and $\{\tau_{j,\e}\}$ the roots of $\,P$ and $\, P_\e$. Therefore there are two constants, depending only on $\,m$, such that, for all $j,k=1,\dotsc,m$ with $\,j\ne k$, it results\textup:
\begin{enumerate}
\item  $\bigl|\tau_{j,\e}(x)-\tau_j(x)\bigr| \mn C_1\, \e$, 
\item  $\bigl|\tau_{k,\e}(x)-\tau_{j,\e}(x)\bigr| \mg C_2\, \e$.
\end{enumerate}
\end{Lemma}

Then we prove:
\begin{Lemma}
Let
\[
\PP_\e(x\pv\tau,\xi)
  \ug  \prod_{j=1}^m\, \bigl(\tau-\tau_{j,\e}(x)\,\xi\bigr)
\]
be a sequence of strictly hyperbolic polynomial satisfying the properties 	\textup{(1)} and \textup{(2)} of Lemma \ref{L-waka}, 
for some constant $C_1$ and $C_2$ independent of~$\e$ and $x$.
\sm

Therefore, if $\, \RR(t,x\pv\tau,\xi)$ has a proper decomposition w.r.t.~$\PP$,
it has also  a proper decomposition w.r.t.~$\PP_\e$,
with coefficients bounded uniformly w.r.t.~$\e$.
\end{Lemma}

\begin{proof}
We know that
\[
\RR(t,x\pv\tau,\xi) \ug  \sum_{k=1}^m \ell_k(x)\PP_{\wh{k}}(\tau)
\]
where the $\, \PP_{\wh{k}}$ are the reduced polynomial of $\,\PP$.
Hence, by linearity, it is sufficient to prove that each  $\, \PP_{\wh{k}}$ has  a proper decomposition w.r.t.~$\PP_\e$.
With no loss of generality we can assume that~$k=1$.
By the Lagrange interpolation formula we have
\[
\PP_{\wh{1}}(\tau)
  \ug  \sum_{k=1}^m \frac{\PP_{\wh{1}}(\tau_{k,\e})}{\PP_{\wh{k},\e}(\tau_{k,\e})} \, \PP_{\wh{k},\e}(\tau) .
\]

If $k=1$, we have
\[
\biggl|\frac{\PP_{\wh{1}}(\tau_{1,\e})}{\PP_{\wh{1},\e}(\tau_{1,\e})}\biggr|
  \ug  \prod_{h=2}^m \Bigl|\frac{\tau_{1,\e}-\tau_h}{\tau_{1,\e}-\tau_{h,\e}}\Bigr|
  \ug  \prod_{h=2}^m \Bigl| 1 + \frac{\tau_{h,\e}-\tau_h}{\tau_{1,\e}-\tau_{h,\e}} \Bigr|
  \mn  \Bigl(1 + \frac{C_1}{C_2}\Bigr)^{m-1} .
\]

If $k\ne1$,  we have
\[
\biggl|\frac{\PP_{\wh{1}}(\tau_{k,\e})}{\PP_{\wh{1},\e}(\tau_{k,\e})}\biggr|
  \ug  \Bigl|\frac{\tau_{k,\e}-\tau_k}{\tau_{k,\e}-\tau_{1,\e}}\Bigr|
     \prod_{\substack{h=2,\dotsc,m\\h\ne k}} \Bigl|\frac{\tau_{k,\e}-\tau_h}{\tau_{k,\e}-\tau_{h,\e}}\Bigr|
  \mn \,  \frac{C_1}{C_2} \, \Bigl(1 + \frac{C_1}{C_2}\Bigr)^{m-2} .
\]
Thus we get the development
\[
\RR(t,x\pv\tau,\xi) \ug  \sum_{k=1}^m \ell_{k,\e}(x)\PP_{\wh{k},\e}(\tau),
\]
with some $ \,\bigl|\ell_{k,\e}(x)\bigr| \mn C$.
\end{proof}


\begin{thebibliography}{KWN}

\bibitem[BB]{BernardiBove-2006}
E. Bernardi and A. Bove,
     \emph{On the~Cauchy problem for some hyperbolic operators
        with double characteristics},
     Phase space analysis of~partial differential equations,
     Progr. Nonlin. Differential Equations Appl. \textbf{69},
     Birch\"{a}user Boston, 2006, 29--44.

\bibitem[B]{Bronstein-1979}
M.~D. Bron\v{s}te\u{\i}n,
     \emph{Smoothness of~roots of~polynomials depending
        on parameters},
     Siberian Math. J. \textbf{20} (1980) 347--352,
     translation from Sibirsk. Mat. Zh. \textbf{20}
     (1979) 493--501.

\bibitem[CO]{ColombiniOrru1999}
F. Colombini and N. Orr\`{u},
     \emph{Well-posedness in $\mathscr{C}^\infty$
        for some weakly hyperbolic equations},
     J.~Math. Kyoto Univ. \textbf{39} (1999), 399--420.

\bibitem[DAS]{DAnconaSpagnolo1998}
P. D'Ancona and S. Spagnolo,
     \emph{Quasi-symmetrization of~hyperbolic systems
        and propagation of~the analytic regularity},
     Boll. Un. Mat. Ital. (8),
     \textbf{B1} (1998), 169--185.

\bibitem[D]{Dunn}
J. L. Dunn,
   \emph{A sufficient condition for hyperbolicity
      of~partial differential operators
      with constant coefficient principal part},
   Trans. Amer. Math Soc. \textbf{201} (1975), 315--327.

\bibitem[F]{Fisk}
S.~Fisk,
   Polynomials, roots, and interlacing,
   2006.
   Avalaible at~\url{https://arxiv.org/abs/math/0612833}.

\bibitem[G]{Garding}
L. G\aa rding,
   \emph{Linear hyperbolic partial differential equations
      with constant coefficients},
   Acta Math. \textbf{85} (1951), 1--62.

\bibitem[GR]{GarettoRuzhansky-2013}
C.~Garetto and M.~Ruzhansky,
   \emph{Weakly hyperbolic equations with non-analytic coefficients 
         and lower order terms},
   Math. Ann. \textbf{357} (2013), 401--440.

\bibitem[H]{Hormander}
L.~Hormander,
   The analysis of~linear partial differential operators,
   Grundlehren der Mathematischen Wissenschaften,
   Springer-Verlag, Berlin, 1985.

\bibitem[J1]{Jannelli1989}
E.~Jannelli,
     \emph{On the~symmetrization of~the principal symbol
        of~hyperbolic equations},
     Comm. Part. Diff. Equat. \textbf{14} (1989), 1617--1634.

\bibitem[J2]{Jannelli2009}
E.~Jannelli,
     \emph{The hyperbolic symmetrizer: theory and applications},
     Advances in phase space analysis of~partial differential equations, pp. 113--139,
     Progr. Nonlinear Differential Equations Appl., 78, Birkh\"{a}user Boston, Inc., Boston, MA, 2009.

\bibitem[JT1]{JannelliTaglialatela-2011}
E.~Jannelli and G.~Taglialatela,
\emph{Homogeneous weakly hyperbolic equations
   with time dependent analytic coefficients},
   J.~Diff.~Eq.~\textbf{251} (2011), 995--1029.

\bibitem[JT2]{JannelliTaglialatela-2014}
E.~Jannelli and G.~Taglialatela,
\emph{Third order homogeneous weakly hyperbolic equations
   with nonanalytic coefficients},
   J. Math. Anal. Appl. 418 (2014), no. 2, 1006--1029.

\bibitem[KS]{KinoshitaSpagnolo}
T. Kinoshita and S. Spagnolo,
     \emph{Hyperbolic equations with non analytic coefficients},
     Math. Ann.~\textbf{336} (2006), 551--569.

\bibitem[M]{Mandai-1985}
T.~Mandai,
     \emph{Smoothness of~roots of~hyperbolic polynomials
        with respect to~one-dimensional parameter},
     Bull. Fac. Gen. Ed. Gifu Univ. \textbf{21} (1985), 115--118.

\bibitem[N1]{Nishitani2020}
T.~Nishitani,
   \emph{Notes on symmetrization by B\'{e}zoutian},
   Boll.~U.M.I. \textbf{13} (2020), 417--428.

\bibitem[N2]{Nishitani2021}
T.~Nishitani,
   \emph{Diagonal symmetrizers for hyperbolic operators with triple characteristics},
   Math. Ann. (2021).

\bibitem[NP1]{NishitaniPetkov1}
T.~Nishitani and V.~Petkov,
   \emph{Cauchy problem for effectively hyperbolic operators with triple characteristics},
   J. Math. Pures Appl. \textbf{123} (2019), 201--228.

\bibitem[NP2]{NishitaniPetkov2}
T.~Nishitani and V.~Petkov,
   \emph{Cauchy problem for hyperbolic operators with triple effective characteristics
         on the initial plane},
   Osaka J. Math. \textbf{57} (2020), 597--615.

\bibitem[N]{Nuij}
W.~Nuij,
     \emph{A note on hyperbolic polynomials},
     Math. Scand. \textbf{23} (1968), 69--72.

\bibitem[P1]{Peyser-1963}
G.~Peyser,
\emph{Energy inequalities for hyperbolic equations
   in several variables with multiple characteristics
   and constant coefficients},
   Trans. Amer. Math. Soc. \textbf{108} (1963), 478--490.

\bibitem[P2]{Peyser-1967}
G.~Peyser,
\emph{On the~roots of~the derivative of~a polynomial
   with real roots},
   Amer. Math. Monthly \textbf{74} (1967), 1102--1104.

\bibitem[Sp]{Spagnolo-2005}
S.~Spagnolo,
     \emph{Hyperbolic systems well-posed
        in all Gevrey classes},
     in Geometric Analysis of~PDE
     and Several Complex Variables
     -- Dedicated to~Fran\c{c}ois Treves,
     S.~Chanillo et al. (Ed.s),
     Contemp. Math. \textbf{368} (2005), 405--414.

\bibitem[ST1]{ST-JHDE}
S.~Spagnolo and G.~Taglialatela,
   \emph{Homogeneous hyperbolic equations
      with coefficients depending on one space variable},
   Journal of~Hyperbolic Diff.~Eq., \textbf{4} (2007), 533--553.
   Erratum 797--797.

\bibitem[ST2]{ST-CPDE}
S.~Spagnolo and G.~Taglialatela,
   \emph{Analytic propagation for nonlinear weakly hyperbolic systems}, 
   Comm. Partial Differential Equations \textbf{35} (2010), 2123--2163. 

\bibitem[Sv]{Svensson}
S. L. Svensson,
  \emph{Necessary and sufficient conditions
     for the~hyperbolicity of~polynomials with
     hyperbolic principal part},
  Ark. Mat. \textbf{8} (1970), 145--162.

\bibitem[Tag]{T-2013}
G.~Taglialatela,
   \emph{Quasi-symmetrizer and Hyperbolic Equations}.
   In: Reissig M., Ruzhansky M. (eds)
   Progress in Partial Differential Equations.
   Springer Proceedings in Mathematics \&~Statistics, vol.~44 (2013).
   Springer, Heidelberg.

\bibitem[Tar]{Tarama-2006}
S.~Tarama,
     \emph{Note on the~Bronshtein theorem concerning
        hyperbolic polynomials},
     Sci. Math. Jpn.  \textbf{63} (2006), 247--285.

\bibitem[W1]{Wakabayashi-1980}
S.~Wakabayashi,
    \emph{The Cauchy problem for operators
       with constant coefficient hyperbolic principal part
       and propagation of~singularities},
    Japan. J. Math. \textbf{6} (1980), 179--228.

\bibitem[W2]{Wakabayashi-1986}
S.~Wakabayashi,
     \emph{Remarks on hyperbolic polynomials},
     Tsukuba J.~Math. \textbf{10} (1986), 17--28.
\end{thebibliography}
\end{document}